\documentclass[a4paper,10pt,reqno]{amsart}
\usepackage{amsmath,amssymb}
\usepackage{url}
\usepackage[all]{xy}

\usepackage{amsmath,amssymb}
\usepackage{url}
\usepackage{supertabular}
\usepackage{multirow}

\newcommand{\Q}[1]{\mathbb{Q}(\sqrt{#1})}
\newcommand{\Z}{\mathbb{Z}}
\newcommand{\N}{\mathbb{N}}

\newcommand{\Fro}{\mathfrak{o}}

\newcommand{\Frp}{\mathfrak{p}}

\newcommand{\Frq}{\mathfrak{q}}

\newcommand{\Fra}{\mathfrak{a}}

\newcommand{\Frn}{\mathfrak{n}}
\newcommand{\Frs}{\mathfrak{s}}
\newcommand{\Frv}{\mathfrak{v}}

\newtheorem{Rmk}{Remark}
\newtheorem{Lem}{Lemma}
\newtheorem{Thm}{Theorem}

\newtheorem{Prop}{Proposition}

\newcommand{\qf}[1]{\langle #1 \rangle}
\newcommand{\conj}[1]{\overline{#1}}
\newcommand{\Case}[1]{\noindent\mbox{\bf Case #1}}
\newcommand{\binlattice}[4]{\begin{pmatrix}%
  #1 & #2 \\
  #3 & #4
\end{pmatrix}}
\newcommand{\terlattice}[9]{{\begin{pmatrix}%
  #1 & #2 & #3 \\
  #4 & #5 & #6 \\
  #7 & #8 & #9
\end{pmatrix}}}

\newcommand{\comega}{\conj\omega}
\newcommand{\nequiv}{\not\equiv}

\newcommand{\gen}{\operatorname{gen}}

\newcommand{\rank}{\operatorname{rank}}

\newcommand{\ra}{\rightarrow}
\newcommand{\nra}{\not\rightarrow}

\newcommand{\univ}{{\!\dag}}

\title[Binary normal regular Hermitian lattices]
{Binary normal regular Hermitian lattices \\over imaginary quadratic fields}%


\author[Byeong Moon Kim]{Byeong Moon Kim}
\address{Department of Mathematics, Kangnung National University, Kangnung, 210-702, Korea}%
\email{kbm@kangnung.ac.kr}

\author[Ji Young Kim]{Ji Young Kim}
\address{Department of Mathematical Sciences, Seoul National University, san 56-1 Shinrim-dong Gwanak-gu, Seoul, 151-747, Korea}%
\email{jykim@math.snu.ac.kr}

\author[Poo-Sung Park]{Poo-Sung Park}
\address{School of Computational Sciences, Korea Institute for Advanced Study, Hoegiro 87, Dongdaemun-gu, Seoul, 130-722, Korea}%
\email{sung@kias.re.kr}

\date{Received: date / Accepted: date}
\begin{document}

\begin{abstract}
We call a positive definite Hermitian lattice regular if it represents all integers, which can be
represented locally by the lattice. We investigate binary regular Hermitian lattices over imaginary
quadratic fields $\Q{-m}$ and provide a complete list of the (normal) binary regular Hermitian
lattices.
\end{abstract}

\maketitle

\section{introduction}\label{sec:introduction}

Mathematicians or number theorists have been pursuing, for a long time, an algorithm to check
whether a given Diophantine equation has a solution without solving the equation directly. A basic
method to do this is to investigate the behavior of the residues modulo some numbers. This method
may show the insolvability of a Diophantine equation. But, it is still considered to be a hard
problem to determine when this method guarantees the solvability.

Dickson first called a positive definite quadratic form $f$ \emph{regular} if $f=n$ has an integral
solution for each $n$ such that $f \equiv n \pmod{m}$ has solutions for all positive integers $m$.
He computed all regular forms $x^2+ay^2+bz^2$, as a generalization of the famous unsolved problem,
Euler's idoneal number $a$ admitting $x^2+ay^2$ to be regular \cite{Dickson}. Jones extended
Dickson's results to the form $ax^2+by^2+cz^2$ \cite{bwJ-31}. His work contained a candidate and it
was immediately solved by himself and Pall \cite{bwJ-gP-39}.

The outstanding result about regular quadratic forms was achieved by Watson. He showed that there
are finitely many equivalence classes of primitive positive definite regular ternary quadratic
forms \cite{gW-53}, \cite{gW-54}. The complete list of 913 regular ternary forms including 22
candidates was given by Jagy, Kaplansky and Schiemann \cite{wJ-iK-aS-97}. On the contrary, Earnest
found an infinite family of regular quaternary forms \cite{agE-95} and the first author carried out
the determination of all regular diagonal quaternary forms \cite{bmK}.

The regularity of integral quadratic forms is naturally generalized to that of lattices over
totally real algebraic number fields. Recently the analogue of Watson's finiteness result for
regular positive definite ternary quadratic lattices over the ring $\Fro$ of $\Q{5}$ was proved
\cite{Chan-Earnest-Icaza-KimJY}.

Regular Hermitian lattices over imaginary quadratic fields are defined in a similar way. If a
Hermitian lattice represents all positive integers, it is trivially regular. We call such lattices
\emph{universal}. The universal Hermitian lattices were concentrative subjects studied by many
mathematicians including the authors in the last couple of decades \cite{agE-aK-97(1)},
\cite{Iwabuchi}, \cite{KimBM-KimJY-Park}, \cite{KimJH-Park}.

The \emph{regular} Hermitian lattices were also investigated and the finiteness of binary \emph{
normal} regular Hermitian lattices was proved by Earnest and Khosravani \cite{agE-aK-97(2)}.
Besides, several binary regular Hermitian diagonal lattices including a candidate $\qf{1, 14}$ over
$\Q{-7}$ were listed by Rokicki \cite{Rokicki}. But the inventory was limited to the diagonal
lattices, that is, $\Fra_1 v_1 \perp \Fra_2 v_2$ with two ideals $\Fra_1, \Fra_2 \subseteq \Fro$
and two vectors $v_1, v_2$. The obstruction against studying Hermitian lattices was that the matrix
presentation was unprovided. The authors, however, developed the formal matrix presentation and
were able to delve into universality and regularity of Hermitian lattices. So we obtained a
complete list of positive definite binary normal regular Hermitian lattices including nondiagonal
lattices with complete proofs. Rokicki's lattice $\qf{1,14}$ over $\Q{-7}$ is proved to be regular.
To do this, we developed a new method to calculate numbers represented by a quaternary quadratic
form. The binary subnormal regular lattices will appear in the next articles.

%%%%%%%%%%%%%%%%%%%%%%%%%%%%%%%%%%%%%%%%%%%%%%%%%%%%%%%%%%%%%%%%%%%%%%%%%%%%%%%%%%%%%%%%
\section{Preliminaries}\label{sec:Preliminaries}

In this section we give some notations and terminologies, which are adopted from \cite{Jacobson}
and \cite{O'Meara}. Let $\Fro$ be the ring of integers of the imaginary quadratic field $E =
\Q{-m}$. We have that $\Fro = \Z[\omega]$ with $\omega = \sqrt{-m}$ if $m \nequiv 3 \pmod{4}$ and
$\omega = \frac{1+\sqrt{-m}}{2}$ if $m \equiv 3 \pmod{4}$. A Hermitian space $V$ is a vector space
over $\Q{-m}$ with a map $H=H_V: V \times V \to \Q{-m}$ satisfying the following conditions:
\begin{enumerate}
    \item $H(v,w) = \conj{H(w,v)}$ for $v, w \in V$,
    \item $H(v_1+v_2, w) = H(v_1,w) + H(v_2,w)$ for $v_1, v_2, w \in V$,
    \item $H(a v,w) = aH(v,w)$ for $a \in \Q{-m}$ and $v, w \in V$.
\end{enumerate}
For brevity, we write $H(v) = H(v,v)$. A Hermitian lattice $L$ is defined as a finitely generated
$\Fro$-module in the Hermitian space $V$. From condition $(1)$, we know that
\[
    H(v) = H(v,v) = \conj{H(v,v)} = \conj{H(v)}.
\]
Hence $H(v)$ is always a rational integer. If $a = H(v)$ for some $v \in L$, we say that $a$ is
represented by $L$ and denote it by $a \ra L$. If $a$ cannot be represented by $L$, we denote it by
$a \nra L$. Through this article, we assume that $L$ is positive definite, i.e., $H(v) > 0$ for
nonzero vectors $v \in L$.

The localization of a lattice $L$ is defined by $L_\Frp = \Fro_\Frp \otimes_\Fro L$ and
$E$-ification of $L$ is defined by $EL = E \otimes_\Fro L$. If $n \ra L_\Frp$ for all primes $\Frp$
including $\infty$, then we write $n \ra \gen{L}$. The regularity of a Hermitian lattice $L$ can be
rephrased as follows: if $n \ra \gen{L}$, then $n \ra L$. Thus if the class number of $L$ is one,
then $L$ is trivially regular.

If a regular lattice $L$ is locally universal over $\Fro_\Frp$ for all prime ideals $\Frp$, then
$L$ is universal. Since all universal Hermitian lattices are already classified
\cite{agE-aK-97(1)}, \cite{Iwabuchi}, \cite{KimJH-Park}, we only consider nonuniversal regular
lattices through this article.

We define two ideals related to values of $H$. The norm $\Frn L$ of $L$ is the $\Fro$-ideal
generated by the set $\{H(v)| v \in L\}$. The scale $\Frs L$ of $L$ is the $\Fro$-ideal generated
by the set $\{H(v,w)| v,w \in L\}$. It is clear that $\Frn L \subseteq \Frs L$. If $\Frn L = \Frs
L$, then we call $L$ \emph{normal}. Otherwise, we call $L$ \emph{subnormal}. We investigate normal
lattices in this article.

The lattice can be written as
\[
    L = \Fra_1 v_1 + \Fra_2 v_2 + \dotsb + \Fra_n v_n
\]
with ideals $\Fra_i \subset \Fro$ and vectors $v_i \in V$. If these vectors are linearly
independent over $\Q{-m}$, then we say that $L$ is a $n$-ary lattice and $\rank L = n$.  The
significant invariant of $L$ is the \emph{volume} defined as
\[
    \Frv L = (\Fra_1 \conj{\Fra}_1) (\Fra_2 \conj{\Fra}_2) \dotsb (\Fra_n\conj{\Fra}_n) \det(H(v_i,v_j))_{1\le i,j \le n}.
\]
Note that this ideal is invariant for equivalent lattices, i.e., $\Frv L = \Frv(\phi L)$ for an
isometry $\phi$ and the volumes of sublattices are contained in $\Frv L$.

If $L$ is a free $\Fro$-module, then we can write $L = \Fro v_1 + \dotsb + \Fro v_n$. The matrix
$M_L = (H(v_i,v_j))_{1\le i,j \le n}$ is called the Gram matrix of $L$ and is a matrix presentation
of $L$. If the matrix is diagonal, we denote it by $\qf{H(v_1),H(v_2),\dotsc,H(v_n)}$. But, if $L$
is not a free $\Fro$-module, then $L= \Fro v_1 + \dotsb + \Fro v_{n-1} + \Fra v_n$ for some ideal
$\Fra \subset \Fro$ \cite[81:5]{O'Meara}. Since any ideal in $\Fro$ is generated by at most two
elements, we can write $L= \Fro v_1 + \dotsb + \Fro v_{n-1} + (\alpha, \beta)\Fro v_n$ for some
$\alpha, \beta \in \Fro$. Therefore, we consider the following $(n+1) \times (n+1)$-matrix as a
formal Gram matrix for $L$:
\[
    M_L =
    \begin{pmatrix}
        H(v_1, v_1)         & \dotsc & H(v_1, \alpha v_n)        & H(v_1, \beta v_n)\\
        \vdots              & \ddots & \vdots                    &\vdots\\
        H(\alpha v_n,  v_1) & \dotsc & H(\alpha v_n, \alpha v_n) & H(\alpha v_n, \beta v_n)\\
        H(\beta  v_n,  v_1) & \dotsc & H(\beta  v_n, \alpha  v_n) & H(\beta  v_n, \beta v_n)\\
    \end{pmatrix}.
\]
Note that this matrix is positive semi-definite, but this represents an $n$-ary positive definite
Hermitian lattice. A scaled lattice $L^a$ obtained from the Hermitian form $H_{L^a} = a H_L$ with
$a \in \Z$. If $M$ is a matrix representation of a lattice $L$, we write $aM$ for the matrix
representation of a scaled lattice $L^a$.

If the (formal) Gram matrix $(a_{ij})$ is called \emph{Minkowski-reduced} if it satisfies the
following conditions:

(1) $a_{ii} \leq a_{jj}$ for $i < j$,

(2) $a_{ii} \leq |2 a_{ij,1}|$ and $a_{ii} \leq |2 a_{ij,2}|$ for $a_{ij} = a_{ij,1} + \omega
a_{ij,2}$ with $i > j$.

A Hermitian lattice is defined over $\Z[\omega]$ and an algebraic integer is of the form $x
+y\omega$. So we can define values of $2n$-ary quadratic form
$\widetilde{H}(x_1,y_1,\dotsc,x_n,y_n)$ over $\Z$ as values of $n$-ary Hermitian form $H(x_1+y_1
\omega, \dotsc, x_n+y_n \omega)$ over $\Z[\omega]$. We call this quadratic form \emph{associated}
with the Hermitian lattice \cite{agE-aK-97(1)}. It is sometimes convenient to consider the
associated quadratic lattice instead of the Hermitian lattice. To distinguish the associated
lattice from the original one, we use the subscript $\Z$. For instance, the quadratic form
$\qf{1,1,1,1}_\Z$ is associated with the Hermitian form $\qf{1,1}$ over $\Z[\sqrt{-1}]$. We abuse
both terms quadratic \emph{forms} and quadratic \emph{lattices}. So we might say that a quadratic
form has a sublattice.

\section{Watson transformation}\label{sec:WatsonT}

For a positive integer $t$ and a Hermitian lattice $L$, let
\[
    \Lambda_t(L)=\{v\in L| H(v,w)\equiv0\pmod t \text{ for all } w\in L\}.
\]
Let the \emph{Watson transformation} $\lambda_t (L)$ of $L$ be the lattice defined by
\[
    \lambda_t(L)=\Lambda_t(L)^{1\over a},
\]
where $a$ is the maximal positive integer which divides all elements of
$\{B(v,w)|v,w\in\Lambda_t(L)\}$. If $L$ is regular, then  $\lambda_t(L)$ is also regular
\cite{gW-62}, \cite{wcC-aR-07}.

Throughout the remainder of this article, $L$ will always means an $\Fro$-lattice on a positive
definite binary Hermitian space over $E=\Q{-m}$, where $m$ is a positive integer. We may assume
that all lattices $L$ under discussion are primitive, since if a lattice is regular, then so is any
multiple of that lattice. Let $p$ be a rational odd prime and $\Frp = \Frp(p)$ be a prime ideal
over $p$ in the ring $\Fro$ of integers. Let $\Frq$ be a dyadic prime ideal of $\Fro$. In
convenience, if $L_{\Frp(p)}$ represents an element $n$ of $\Z_p$ then we say that $n \ra L$ over
$\Z_p$. In the dyadic case, we say that $n \ra L$ over $\Z_2$. The following lemma will be useful
for later discussion.

\begin{Lem}\label{Lem:local_universality_of_Watson_transformation}
Let $L$ be a primitive normal binary Hermitian lattice. Then $L_\Frp$ is isometric to
$\qf{\epsilon, \epsilon'p^k}$ for some nonnegative integer $k$ and units $\epsilon$, $\epsilon'$ of
$\Fro_\Frp$. Moreover, $\lambda_{p^{k'}}(L_\Frp)$ represents all elements of $\Z_p$ for some
nonnegative integer $k'$.
\end{Lem}

\begin{proof}
Since $L$ is primitive normal, $L_\Frp$ is isometric to $\qf{\epsilon,\epsilon'p^k}$ for some
nonnegative integer $k$ and units $\epsilon$, $\epsilon'$ of $\Fro_\Frp$. We may assume that any
unary $\Fro_\Frp$-lattice is not isotropic. Otherwise $L_\Frp$ represents all elements of $\Z_p$.

If $(p,m)=1$ and $\Frp=p\Fro$, $k=2\ell+r$ for some $\ell\in\Z$ and $r=0,1$, then,
$\lambda_{p^{2\ell}}(L_\Frp)$ is isometric to $\qf{\epsilon, \epsilon'p^r}$. Since the Hermitian
lattice $\qf{\epsilon,\epsilon'p^r}$ over $\Fro_\Frp$ provides the associated quadratic lattice
$\qf{\epsilon,\epsilon m,\epsilon'p^r,\epsilon'mp^r}_{\Z_p}$ $\cong$ $\qf{1,m,p^r,mp^r}_{\Z_p}$
over $\Z_p$, $\lambda_{p^{2\ell}}(L_\Frp)$ represents all elements of $\Z_p$.

If $(p,m)=1$ and $\Frp \ne p\Fro$, then $\lambda_{p^{k-1}}(L_\Frp)$ is isometric to $\qf{\epsilon,
\epsilon'p}$. Since the Hermitian lattice $\qf{\epsilon, \epsilon'p}$ provides the associated
quadratic lattice $\qf{\epsilon, \epsilon m,\epsilon'p, \epsilon'mp}_{\Z_p}$ $\cong$ $\qf{1,
m,p,mp}_{\Z_p}$ over $\Z_p$, $\lambda_{p^{k-1}}(L_\Frp)$ represents all elements of $\Z_p$.

If $p|m$, then $\Frp \ne p\Fro$. So $\lambda_{p^{k}}(L_\Frp)$ is isometric to $\qf{\epsilon,
\epsilon'}$. Since the Hermitian lattice $\qf{\epsilon,\epsilon'}$ over $\Fro_\Frp$ provides the
associated quadratic lattice $\qf{\epsilon,\epsilon m,\epsilon',\epsilon'm}_{\Z_p}$ $\cong$
$\qf{1,\epsilon \epsilon',p, \epsilon\epsilon'p}_{\Z_p}$ over $\Z_p$, $\lambda_{p^{k}}(L_\Frp)$
represents all elements of $\Z_p$.
\end{proof}

\begin{Lem}
Let $L$ be a primitive normal binary Hermitian lattice. Then $L_\Frq$ is isometric to
$\qf{\epsilon, \epsilon'2^k}$ for some nonnegative integer $k$ and units $\epsilon$, $\epsilon'$ of
$\Fro_\Frq$. Moreover, $\lambda_{2^{k'}}(L_\Frq)$ represents all elements of $\Z_p$ for some
nonnegative integer $k'$.
\end{Lem}

\begin{proof}
Since $L$ is primitive normal, $L_{\Frq}$ is isometric to $\qf{\epsilon, \epsilon'2^k}$ for some
nonnegative integer $k$ and units $\epsilon$, $\epsilon'$ of $\Fro_\Frq$.

If $m \equiv 7 \pmod8$, then the unary Hermitian lattice $\qf{\epsilon}$ over $\Fro_\Frq$ provides
the associated quadratic lattice $\binlattice{\epsilon}{\epsilon / 2}{\epsilon / 2}{(m+1)\epsilon
/4}_{\Z_2} \cong \binlattice{0}{1/2}{1/2}{0}_{\Z_2}$ over $\Z_2$. It represents all elements of
$\Z_2$. So does $L_{\Frq}$.

If $m \equiv 3 \pmod8$, let $k=2\ell+r$ for some $\ell\in\Z$ and $r=0,1$. Then
$\lambda_{4^{\ell}}(L_{\Frq})$ is isometric to $\qf{\epsilon,\epsilon'2^r}$. Since the Hermitian
lattice $\qf{\epsilon,\epsilon'2^r}$ over $\Fro_\Frq$ provides the associated quadratic lattice
$\binlattice{\epsilon}{\epsilon / 2}{\epsilon / 2}{(m+1)\epsilon/4}_{\Z_2} \perp
\binlattice{2^r\epsilon'}{\epsilon'2^{r-1}}{\epsilon'2^{r-1}}{{(m+1)2^{r-2}}\epsilon'}_{\Z_2}$
which is isometric to $\binlattice{1}{1/2}{1/2}{1}_{\Z_2} \perp \binlattice{1}{1/2}{1/2}{1}_{\Z_2}$
or $\binlattice{1}{1/2}{1/2}{1}_{\Z_2} \perp \binlattice2112_{\Z_2}$ over $\Z_2$,
$\lambda_{4^{\ell}}(L_{\Frq})$ represents all elements of $\Z_2$.

If $m \equiv 1 \pmod4$, let $k=2\ell+r$ for some $\ell\in\Z$ and $r=0,1$. Then
$\lambda_{4^{\ell}}(L_{\Frq})$ is isometric to $\qf{\epsilon,\epsilon'2^r}$. Since the Hermitian
lattice $\qf{\epsilon,\epsilon'2^r}$ over $\Fro_\Frq$ provides the associated quadratic lattice
$\qf{\epsilon, \epsilon m, 2^r\epsilon', 2^r\epsilon'm}_{\Z_2}$ over $\Z_2$,
$\lambda_{4^{\ell}}(L_{\Frq})$ represents all elements of $\Z_2$.

If $m \equiv 2 \pmod4$, then $\lambda_{2^{k}}(L_{\Frq})$ is isometric to $\qf{\epsilon,\epsilon'}$.
If $m=2m'$, Hermitian lattice $\qf{\epsilon,\epsilon'}$ over $\Fro_\Frq$ provides the associated
quadratic lattice $\qf{\epsilon,\epsilon m,\epsilon',\epsilon'm}_{\Z_2}$ over $\Z_2$. Since
quadratic lattice $\qf{\epsilon,\epsilon m,\epsilon',\epsilon'm}_{\Z_2}$ is isometric to
$\qf{\epsilon,\epsilon',2m'\epsilon,2m'\epsilon'}_{\Z_2}$, $\lambda_{2^{k}}(L_{\Frq})$ represents
all elements of $\Z_2$.
\end{proof}

Thus for all rational prime number $p$ including $2$, there is a nonnegative integer $s$ such that
$\lambda_{p^s}(L)$ represents all elements of $\Z_p$. So there are primes $p_1, p_2, \cdots, p_k$
and positive integers $s_1, s_2, \cdots, s_k$ such that $\widehat{L} = \lambda_{p_1^{s_1}} \circ
\lambda_{p_2^{s_2}} \circ \cdots \circ \lambda_{p_k^{s_k}} (L)$ is locally universal, which means
$\widehat{L}$ represents all elements of $\Z_p$ for all prime $p$. Since $\widehat{L}$ is regular,
$\widehat{L}$ is universal. From the works on binary universal Hermitian lattices
\cite{agE-aK-97(1)}, \cite{Iwabuchi}, \cite{KimJH-Park}, we have the following proposition.

\begin{Prop}
A binary normal regular lattice exists over the field $\Q{-m}$ if and only if $m$ is
\[
    1, 2, 3, 5, 6, 7, 10, 11, 15, 19, 23 \text{ or } 31.
\]
\end{Prop}

Note that the class number $h(\Q{-m})$ is one for $m = 1, 2, 3, 7, 11, 19$ and  $h(\Q{-m}) > 1$ for
$m = 5, 6, 10, 15, 23, 31$.

\section{Candidates of binary normal regular Hermitian lattices}\label{sec:Candidates}

In this section, we will find all candidates of binary normal regular Hermitian lattices over
imaginary quadratic fields $\Q{-m}$.

If there is a rational odd prime number $p_0$ such that $(p_0,m)=1$ and $L_{\Frp_0}$ does not
represent some element of $\Z_{p_0}$ where $\Frp_0 = \Frp_0(p_0)$ is a prime ideal of $\Fro$. Let
$p_1,p_2,\cdots,p_k$ be all rational prime numbers different from $p_0$ such that $L_{\Frp_i}$ does
not represent some element of $\Z_{p_i}$. Then there are positive integers $s_1,s_2,\cdots,s_k$
such that for all $i=1,2,\cdots,k$, $\lambda_{p_i^{s_i}}(L_{\Frp_i})$ represents all elements of
$\Z_{p_i}$. Let
\[
    \widehat{L} = \lambda_{p_1^{s_1}}\circ\lambda_{p_2^{s_2}}\circ\cdots\circ\lambda_{p_k^{s_k}}(L).
\]
Then $\widehat{L}$ is regular. So $\widehat{L}_\Frp$ represents all elements of $\Z_p$ for all
prime numbers $p$ and prime ideals $\Frp$ except $p = p_0$. Since $\widehat{L}_{\Frp_0}$ is
primitive and normal, $\widehat{L}_{\Frp_0}$ is isometric to $\qf{\epsilon, \epsilon'p_0^k}$ for
some positive integer $k$ and units $\epsilon, \epsilon'$ of $\Fro_{p_0}$. Since $(p_0,m)=1$,
$\qf{\epsilon}$ represents all units of $\Z_{p_0}$. So $\widehat{L}$ represents $1$ and $2$
locally. Since $L$ is regular, $\widehat{L}$ represents $1$ and $2$ globally. So $\widehat{L}$ is
isometric to $\qf{1} \perp M$ for some unary lattice $M$. If $m \not= 1, 2, 7$, then $\qf{1}$ does
not represent $2$. So $M$ represents $1$ or $2$. Thus $\widehat{L}$ contains $\qf{1,1}$ or
$\qf{1,2}$. Therefore $\widehat{L}$ represents all elements of $\Z_{p_0}$. This is a contradiction.
If $m = 1$ or $7$, then $\qf{1}$ cannot represent $3$. Since $\widehat{L}$ is regular, $3$ is not a
unit of $\Z_{p_0}$. So $p_0 = 3$. Similarly, if $m=2$, then $p_0 = 5$. We conclude that if $L_\Frp$
does not represent some element of $\Z_p$, then we have following cases:%
\begin{enumerate}
  \item $p = 2$
  \item odd prime $p$ divides $m$
  \item $\begin{cases}p=3 & \text{if } m=1,7, \\ p=5 & \text{if } m=2.
  \end{cases}$
\end{enumerate}

To find candidates of regular lattices with efficiency, we add a condition of volume $\Frv L$ of
$L$ as explained in the following lemma.

\begin{Lem}
Let $L$ be a binary Hermitian lattice over the imaginary quadratic field $\Q{-m}$. Let $p$ be a
rational odd prime and $\Frp = \Frp(p)$ be a prime ideal over $p$ in the ring $\Fro$ of integers
and let $\Frq$ be a dyadic prime ideal of $\Fro$.
    \begin{enumerate}
    \item[(1)] If $L_\Frp$ represents a unit in $\Z_p$ over $\Fro_\Frp$ and
    does not represent $p^k \epsilon$ for some nonnegative integer $k$, for some unit $\epsilon$
    in $\Z_p$ over $\Fro_\Frp$, then
    \[
    \Frv L \subset p^{k+1} \Fro.
    \]
    \item[(2)] If $L_\Frq$ represents a unit in $\Z_2$ over $\Fro_\Frq$ and
    does not represent $2^k \epsilon$ for some nonnegative integer $k$,
    for some unit $\epsilon$ in $\Z_2$ over $\Fro_\Frq$, then
        \begin{align*}
            \begin{cases}
            \Frv L \subset 2^{k+2} \Fro & \text{ if }~ m \equiv 1 \pmod{4}, \\
            \Frv L \subset 2^{k+3} \Fro & \text{ if }~ m \equiv 2 \pmod{4}, \\
            \Frv L \subset 2^{k+1} \Fro & \text{ if }~ m \equiv 3 \pmod{8}.
            \end{cases}
        \end{align*}
    \end{enumerate}
\end{Lem}

\begin{proof}
(1) Since $L_\Frp$ represents a unit, $L_\Frp \cong \qf{a,bp^\ell}$ for some units $a, b \in \Z_p$
and some nonnegative integer $\ell$. If $\qf{a}$ is isotropic, then $p^k\epsilon \ra \qf{a}$ and
hence $p^k\epsilon \ra L_\Frp$. This is a contradiction. Therefore $\qf{a}$ is anisotropic. Assume
that $p \nmid m$. Then the associated quadratic lattice of $L_\Frp$ is isometric to
$\qf{a,a',bp^\ell,b'p^\ell}_{\Z_p}$ for some units $a',b' \in \Z_p$. If $\ell \le k$, then
$p^k\epsilon \ra L_\Frp$. Now assume that $p \mid m$. Then the associated quadratic lattice is
$\qf{a,a'p,bp^\ell,b'p^{\ell+1}}_\Z$. If $\ell \le k$, then $p^k\epsilon \ra L_\Frp$. Thus $\ell
\ge k+1$ and $\Frv L_\Frp = abp^\ell\Fro_\Frp \subset p^{k+1}\Fro_\Frp$.

(2) Since $L_\Frq$ represents a unit in $\Z_2$, $L_\Frq$ is isometric to $\qf{a, 2^\ell b}$ for
some units $a, b \in \Z_2$ and for some integer $\ell$.

Suppose $m \equiv 1 \pmod{4}$. If $\ell= 0, 1$ then $L_\Frq = \qf{a, 2^\ell b}$ represents all
elements of $\Z_2$. Hence we have $\ell \geq 2$. Since $\lambda_{2^{\ell-1}} (L_\Frq) = \qf{a, 2b}$
represents all elements of $\Z_2$, $2^{k-\ell+1} \epsilon \ra \lambda_{2^{\ell-1}} (L_\Frq) =
\qf{a, 2b}$ if $\ell \leq k+1$. Hence $(2^{k-\ell+1} \epsilon) 2^{\ell-1} = \epsilon 2^k \ra
L_\Frq$, which is a contradiction. So $\ell \geq k+2$ and $\Frv{L_\Frq} = ab2^\ell\Fro_\Frq \subset
2^{k+2}\Fro_\Frq$.

Suppose $m \equiv 2 \pmod{4}$. If $\ell= 0, 1, 2$ then $L_\Frq = \qf{a, 2^\ell b}$ represents all
elements of $\Z_2$. Hence we have $\ell \geq 3$. Since $\lambda_{2^{\ell-2}} (L_\Frq) = \qf{a, 4b}$
represents all elements of $\Z_2$, $2^{k-\ell+2} \epsilon \ra \lambda_{2^{\ell-2}} (L_\Frq) =
\qf{a, 4b}$ if $\ell \leq k+2$. Hence $(2^{k-\ell+2} \epsilon) 2^{\ell-2} = \epsilon 2^k \ra
L_\Frq$, which is a contradiction. So $\ell \geq k+3$ and $\Frv{L_\Frq} = ab2^\ell\Fro_\Frq \subset
2^{k+3}\Fro_\Frq$.

Suppose $m \equiv 3 \pmod{8}$. If $\ell= 0, 1$, then $L_\Frq = \qf{a, 2^\ell b}$ represents all
elements of $\Z_2$. Hence we have $\ell \geq 2$. Since $\lambda_{2^\ell} (L_\Frq) = \qf{a, b}$ or
$\qf{2a, b}$ which represents all elements of $\Z_2$, $2^{k-\ell} \epsilon \ra \lambda_{2^\ell}
(L_\Frq)$, if $\ell \leq k$. Hence $(2^{k-\ell} \epsilon) 2^\ell = \epsilon 2^k \ra L$, which is a
contradiction. So $\ell \geq k+1$ and $\Frv{L_\Frq} = ab2^\ell\Fro \subset 2^{k+1}\Fro_\Frq$.

\end{proof}

We adopt some notations from Conway-Sloane \cite{jhC-00}. The notation $p^d$ (resp. $p^e$) denotes
an odd (resp. even) power of $p$; if $p=2$, $u_k$ denotes a unit of form $8n + k ( k = 1,3,5,7 )$
and if $p$ is odd, $u_{+}$(resp. $u_{-}$) denotes a unit which is a quadratic residue (resp.
nonresidue) modulo $p$. Let
    \[
    [a, \alpha, b] := \binlattice{a}{\alpha}{\conj\alpha}{b}
    \]
for simplicity.

From now on, we assume $L$ is a binary normal regular Hermitian lattice which is not universal over
the imaginary quadratic field $\Q{-m}$. We begin with finding candidates of $L$ with the
information of $L_\Frp$, $L_\Frq$ and the volume $\Frv L$ of $L$ by following strategy: We assume
that $a$ is the minimum number such that $a \ra \gen L$ and $b$ is the minimum number such that $b
\ra \gen L$ and $b \nra \qf{a}$. Then $L$ contains a lattice
$\ell=\binlattice{a}{\alpha}{\conj{\alpha}}{b}$ for some $\alpha \in \Fro$ and also $\Frv \ell =
(ab-{\alpha}{\conj{\alpha}}) \Fro \subset \Frv L$. We call these two numbers $a$ and $b$
\emph{essential numbers}. If $\Frv \ell = \Frv L$, then $\ell$ is a candidate of $L$. If $\Frv \ell
\subsetneq \Frv L$, then we do more escalation which satisfies volume and rank conditions. Because
most of the finding process are routine, we will give tables instead of describing in detail except
special cases. In the table, all lattices are Minkowski-reduced. When we show that a binary
Hermitian lattice $L$ is not regular, we will give an integer $n$ such that $n \ra \gen L$ but $n
\nra L$. This number is called the \emph{exceptional number} of $L$.\\

\Case{I} $(m, p) = 1$.

%%%%%%%%%%%%%%%%%%%%%%%%%%%%%%%%%%%%%%%%%%%%%%%%%%%%%%%%%%%%%%%%%%%%%%%%%%%%%%
\Case{I [$m=1$]} Note that
\begin{gather*}
  u_1 \ra L \Longleftrightarrow u_5 \ra L \Longrightarrow 2u_1, 2u_5 \ra L
  \text{ over } \Z_2;\\
  u_3 \ra L \Longleftrightarrow u_7 \ra L \Longrightarrow 2u_3, 2u_7 \ra L
  \text{ over } \Z_2; \\
  u_{+} \ra L \Longleftrightarrow u_{-} \ra L \text{ over } \Z_3.
\end{gather*}

Since $L$ is normal, $L$ represents a number in $\Z_2^\times \cap \Z_3^\times$. Since $L$ is not
universal, $L$ cannot represent all elements of $\Z_2$ and $\Z_3$. According to the
representability of $u_1, u_3$ over $\Z_2$ and $3u_{+}$ over $\Z_3$, we have five cases (See Table
\ref{tbl:m=1}).

\begin{footnotesize}
\begin{table}[!h] % m=1
\centering

\begin{tabular}{p{0.9cm}p{4cm}|p{1.6cm}|p{0.8cm}|p{2.5cm}|p{0.6cm}} \hline%
\multicolumn{2}{c|}{Local Condition}                & \multicolumn{1}{c|}{Volume} & Ess.\#  & \multicolumn{1}{c|}{Reduced Lattice} & Exc.\# \\ \hline%

(1)   & $u_1, u_3 \ra L$ over $\Z_2$                & $\Frv L \subset 3^2 \Fro$ & $1,7$     & N.A.                &      \\ %
      & $3u_{+} \nra L$ over $\Z_3$                 &                           &           &                     &      \\ \hline%
(2)   & $u_1 \ra L$, $u_3 \nra L$ over $\Z_2$       & $\Frv L \subset 2^2\Fro$  & $1,21$    & $\qf{1,4}$          & none \\%
      & $3u_{+} \ra L$ over $\Z_3$                  &                           &           & $\qf{1,8}$          & none \\%
      &                                             &                           &           & $\qf{1,12}$         & 6    \\%
      &                                             &                           &           & $\qf{1,16}$         & none \\%
      &                                             &                           &           & $\qf{1,20}$         & 6    \\ \hline %
(3)   & $u_1 \ra L$, $u_3 \nra L$ over $\Z_2$       & $\Frv L \subset 2^2 \cdot 3^2\Fro$ %
                                                                                & $1,77$    & $\qf{1,36}$         & 14   \\ %
      & $3u_{+} \nra L$ over $\Z_3$                 &                           &           & $\qf{1,72}$         & 28   \\ \hline%
(4-1) & $u_1 \nra L$, $u_3, 2u_1 \ra L$ over $\Z_2$ & $\Frv L \subset 2^2\Fro$  & $2,3$     & $[2, -1+\omega, 3]$ & none \\%
      & $3u_{+} \ra L$ over $\Z_3$                  &                           &           &                     &      \\ \hline%
(4-2) & $u_1, 2u_1 \nra L$, $u_3 \ra L$ over $\Z_2$ & $\Frv L \subset 2^3\Fro$  & $3,7$     & $[3, -1+\omega, 6]$ & none \\ %
      & $3u_{+} \ra L$ over $\Z_3$                  &                           &                     & $[3,1,3]$           & none \\ \hline%
(5-1) & $u_1 \nra L$, $2u_1, u_3 \ra L$ over $\Z_2$ & $\Frv L \subset 2^2 \cdot 3^2\Fro$%
                                                                                & $2,7$     & N.A.                &      \\%
      & $3u_{+} \nra L$ over $\Z_3$                 &                           &           &                     &      \\ \hline%
(5-2) & $u_1, 2u_1 \nra L$, $u_3 \ra L$ over $\Z_2$ & $\Frv L \subset 2^3\cdot 3^2\Fro$%
                                                                                & $7,11$    & $[7,-2+\omega,11] $ & 4    \\ %
      & $3u_{+} \nra L$ over $\Z_3$                 &                           &           &                     &      \\ \hline %
\end{tabular}
\caption{escalation when $m=1$} %
\label{tbl:m=1}%%
\end{table}
\end{footnotesize}

%%%%%%%%%%%%%%%%%%%%%%%%%%%%%%%%%%%%%%%%%%%%%%%%%%%%%%%%%%%%%%%%%%%%%%%%%%%%%%
\Case{I [$m=2$]} Note that
\begin{gather*}
  u_1 \ra L \Longleftrightarrow u_3 \ra L \Longrightarrow 2u_1, 2u_3 \ra L
  \text{ over } \Z_2; \\
  u_5 \ra L \Longleftrightarrow u_7 \ra L \Longrightarrow 2u_5, 2u_7 \ra L
  \text{ over } \Z_2; \\
  u_{+} \ra L \Longleftrightarrow u_{-} \ra L \text{ over } \Z_5.
\end{gather*}
According to the representability of $u_1, u_5$ over $\Z_2$ and $5u_{+}$ over $\Z_5$, we have five
cases (See Table \ref{tbl:m=2}).

\begin{footnotesize}
\begin{table}[!h]% m=2
\centering

\begin{tabular}{p{0.9cm}p{4cm}|p{1.6cm}|p{0.8cm}|p{2.5cm}|p{0.6cm}} \hline%
\multicolumn{2}{c|}{Local Condition}                & \multicolumn{1}{c|}{Volume} & Ess.\#  & \multicolumn{1}{c|}{Reduced Lattice} & Exc.\# \\ \hline%

(1)   & $u_1, u_5 \ra L$ over $\Z_2$                & $\Frv L \subset 5^2 \Fro$   & $1,7$   & N.A.                &      \\ %
      & $5u_{+} \nra L$ over $\Z_5$                 &                             &         &                     &      \\ \hline%
(2)   & $u_1 \ra L$, $u_5 \nra L$ over $\Z_2$       & $\Frv L \subset 2^3 \cdot 5^2\Fro$  %
                                                                                  & $1,91$  & N.A.                &      \\%
      & $5u_{+} \nra L$ over $\Z_5$                 &                             &         &                     &      \\ \hline%
(3-1) & $u_1, 2u_5 \ra L$, $u_5 \nra L$ over $\Z_2$ & $\Frv L \subset 2^3\Fro$    & $1,10$  & $\qf{1,8 }$         & none \\ %
      & $5u_{+} \ra L$ over $\Z_5$                  &                             &         &                     &      \\ \hline%
(3-2) & $u_1 \ra L$, $u_5, 2u_5 \nra L$ over $\Z_2$ & $\Frv L \subset 2^4\Fro$    & $1,35$  & $\qf{1,16}$         & none \\%
      & $5u_{+} \ra L$ over $\Z_5$                  &                             &         & $\qf{1,32}$         & none \\ \hline%
(4)   & $u_1 \nra L$, $u_5 \ra L$ over $\Z_2$       & $\Frv L \subset 2^3\Fro$    & $5,7$   & $[5, -1+\omega, 7]$ & 8  \\ %
      & $5u_{+} \ra L$ over $\Z_5$                  &                             &         & $[5, -2+\omega, 6]$ & 2 \\ %
      &                                             &                             &         & $[5, -1+2\omega, 5]$& 4 \\ %
      &                                             &                             &         & $[4, -2+4\omega, 5]$& 2 \\ %
      &                                             &                             &         & $[2, \omega, 5]$    & none \\ \hline%
(5)   & $u_1 \nra L$, $u_5 \ra L$ over $\Z_2$       & $\Frv L \subset 2^3 \cdot 5^2 \Fro$ %
                                                                                  & $7,13$  & N.A.                &      \\%
      & $5u_{+} \nra L$ over $\Z_5$                 &                             &         &                     &      \\ \hline%
\end{tabular}
\caption{escalation when $m=2$} %
\label{tbl:m=2}%%
\end{table}
\end{footnotesize}

%%%%%%%%%%%%%%%%%%%%%%%%%%%%%%%%%%%%%%%%%%%%%%%%%%%%%%%%%%%%%%%%%%%%%%%%%%%%%%
\Case{I [$m=7$]} Note that
\begin{gather*}
  u_1 \ra L \Longleftrightarrow u_3 \ra L
  \Longleftrightarrow u_5 \ra L \Longleftrightarrow u_7 \ra L
  \Longrightarrow 2u_1, 2u_3, 2u_5, 2u_7 \ra L
  \text{ over } \Z_2; \\
  u_{+} \ra L \Longleftrightarrow u_{-} \ra L \text{ over } \Z_3.
\end{gather*}

So $3u_{+} \nra L$ over $\Z_3$. From the above conditions, $1, 5 \ra L$ and $\Frv L \subset 3^2
\Fro$. But no lattice satisfies the volume condition. \\

%%%%%%%%%%%%%%%%%%%%%%%%%%%%%%%%%%%%%%%%%%%%%%%%%%%%%%%%%%%%%%%%%%%%%%%%%%%%%%%%%%%
\Case{II} $(m, p) \ne 1$.

%%%%%%%%%%%%%%%%%%%%%%%%%%%%%%%%%%%%%%%%%%%%%%%%%%%%%%%%%%%%%%%%%%%%%%%%%%%%%%
\Case{II [$m=3$]} Note that
\begin{gather*}
  u_1 \ra L \Longleftrightarrow u_3 \ra L
  \Longleftrightarrow u_5 \ra L \Longleftrightarrow u_7 \ra L
  \text{ over } \Z_2; \\
  u_{+} \ra L \Longrightarrow 3u_{+} \ra L \text{ over } \Z_3;\quad
  u_{-} \ra L \Longrightarrow 3u_{-} \ra L \text{ over } \Z_3.
\end{gather*}
According to the representability of $2u_1$ over $\Z_2$ and $u_{+}, u_{-}$ over $\Z_3$, we have
five cases (See Table \ref{tbl:m=3}).

\begin{footnotesize}
\begin{table}[!h]
\centering

\begin{tabular}{p{0.9cm}p{4cm}|p{1.6cm}|p{0.8cm}|p{2.5cm}|p{0.6cm}} \hline%
\multicolumn{2}{c|}{Local Condition}                & \multicolumn{1}{c|}{Volume} & Ess.\#  & \multicolumn{1}{c|}{Reduced Lattice} & Exc.\# \\ \hline%

(1)   & $2u_1 \ra L$ over $\Z_2$                          & $\Frv L \subset 3\Fro$      & $1,10$  & $\qf{1,3}$          & none \\ %
      & $u_{+} \ra L$, $u_{-} \nra L$ over $\Z_3$         &                             &         & $\qf{1,6}$          & none \\ %
      &                                                   &                             &         & $\qf{1,9}$          & none \\ \hline%
(2-1) & $2u_1 \ra L$ over $\Z_2$                          & $\Frv L \subset 3\Fro$      & $2,3 $  & $\qf{2,3}$          & none \\%
      & $u_{+} \nra L$, $u_{-}, 3u_{+} \ra L$ over $\Z_3$ &                             &         & $[2,1,2]$           & none \\ \hline%
(2-2) & $2u_1 \ra L$ over $\Z_2$                          & $\Frv L \subset 3^2\Fro$    & $2,5 $  & $[2,1,5]$           & none \\ %
      & $u_{+}, 3u_{+} \nra L$, $u_{-} \ra L$ over $\Z_3$ &                             &         &                     &      \\ \hline %
(3)   & $2u_1 \nra L$ over $\Z_2$                         & $\Frv L \subset 2^2\Fro$    & $1,10$  & $\qf{1,4}$          & none \\%
      & $u_{+}, u_{-} \ra L$ over $\Z_3$                  &                             &         &                     &      \\ \hline%
(4)   & $2u_1 \nra L$ over $\Z_2$                & $\Frv L \subset 2^2 \cdot 3\Fro$     & $1,55$  & $\qf{1,12}$         & none \\ %
      & $u_{+} \ra L$, $u_{-} \nra L$ over $\Z_3$         &                             &         & $\qf{1,24}$         & 15   \\ %
      &                                                   &                             &         & $\qf{1,36}$         & none \\%
      &                                                   &                             &         & $\qf{1,48}$         & 15   \\ \hline%
(5-1) & $2u_1 \nra L$ over $\Z_2$                & $\Frv L \subset 2^2 \cdot 3\Fro$     & $3,5 $  & $[3, 1+\omega, 5]$  & none \\ %
      & $u_{+} \nra L$, $u_{-}, 3u_{+} \ra L$ over $\Z_3$ &                             &         &                     &      \\ \hline%
(5-2) & $2u_1 \nra L$ over $\Z_2$                & $\Frv L \subset 2^2 \cdot 3^2\Fro$   & $5,11$  & $[5,2,8]$           & none \\%
      & $u_{+}, 3u_{+} \nra L$, $u_{-} \ra L$ over $\Z_3$ &                             &         &                     &     \\ \hline%
\end{tabular}
\caption{escalation when $m=3$} %
\label{tbl:m=3}%%
\end{table}
\end{footnotesize}

%%%%%%%%%%%%%%%%%%%%%%%%%%%%%%%%%%%%%%%%%%%%%%%%%%%%%%%%%%%%%%%%%%%%%%%%%%%%%%
\Case{II [$m=5$]} Note that
\begin{gather*}
  u_1 \ra L \Longleftrightarrow u_5 \ra L
  \Longrightarrow 2u_3, 2u_7 \ra L
  \text{ over } \Z_2; \\
  u_3 \ra L \Longleftrightarrow u_7 \ra L
  \Longrightarrow 2u_1, 2u_5 \ra L
  \text{ over } \Z_2; \\
  u_{+} \ra L \Longrightarrow 5u_{+} \ra L \text{ over } \Z_5;\quad
  u_{-} \ra L \Longrightarrow 5u_{-} \ra L \text{ over } \Z_5.
\end{gather*}
According to the representability of $u_1, u_3$ over $\Z_2$ and
 $u_{+}, u_{-}$ over $\Z_5$, we have eight
cases (See Table \ref{tbl:m=5}). Since the ring $\Fro$ of integers is not a PID, we should also
consider nonfree lattices.

\begin{footnotesize}
\begin{table}[!h] %m=5
\centering

\begin{tabular}{p{0.9cm}p{4cm}|p{1.6cm}|p{0.8cm}|p{2.5cm}|p{0.6cm}} \hline%
\multicolumn{2}{c|}{Local Condition}                & \multicolumn{1}{c|}{Volume} & Ess.\#  & \multicolumn{1}{c|}{Reduced Lattice} & Exc.\# \\ \hline%

(1)     & $u_1, u_3 \ra L$ over $\Z_2$                      & $\Frv L \subset 5\Fro$    & $1,11$   & $\qf{1,5}$          & 15   \\       %
        & $u_{+} \ra L$, $u_{-} \nra L$ over $\Z_5$         &                           &          & $\qf{1,10}$         & none \\       %
        &                                                   &                           &          & $\qf{1}\perp 5[2,-1+\omega,3]$   %
                                                                                                                         & none \\ \hline%
(2-1)   & $u_1, u_3 \ra L$ over $\Z_2$                      & $\Frv L \subset 5\Fro$    & $2,3 $   & $[2,1,3]$           & 11   \\       %
        & $u_{+} \nra L$, $u_{-}, 5u_{+} \ra L$ over $\Z_5$ &                           &          & $[2, -1+\omega, 3] \perp \qf{5}$    %
                                                                                                                         & none \\ \hline%
(2-2)   & $u_1, u_3 \ra L$ over $\Z_2$                      & $\Frv L \subset 5^2\Fro$  & $2,3 $   & N.A.                &      \\       %
        & $u_{+}, 5u_{+} \nra L$, $u_{-} \ra L$ over $\Z_5$ &                           &          &                     &      \\ \hline%
(3-1)   & $u_1, 2u_1 \ra L$, $u_3 \nra \Z_2$ over $\Z_2$    & $\Frv L \subset 2^3 \Fro$ & $1,2 $   & N.A.                &      \\       %
        & $u_{+}, u_{-} \ra L$ over $\Z_5$                  &                           &          &                     &      \\ \hline%
(3-2)   & $u_1 \ra L$, $u_3, 2u_1 \nra \Z_2$ over $\Z_2$    & $\Frv L \subset 2^3\Fro$  & $1,13$   & $\qf{1,8}$          & none \\       %
        & $u_{+}, u_{-} \ra L$ over $\Z_5$                  &                           &          &                     &      \\ \hline%
(4-1)   & $u_1, 2u_1 \ra L$, $u_3 \nra L$ over $\Z_2$       & \multicolumn{4}{c}{N.A.}                                          \\       %
        & $u_{+} \ra L$, $u_{-} \nra L$ over $\Z_5$         & \multicolumn{4}{c}{}                                              \\ \hline%
(4-2-1) & $u_1 \ra L$, $u_3, 2u_1 \nra L$ over $\Z_2$       & $\Frv L \subset 2^3\cdot 5\Fro$                                            %
                                                                                        & $1,65$   & $\qf{1, 40}$        & none \\       %
        & $u_{+}, 5u_{-} \ra L$, $u_{-} \nra L$ over $\Z_5$ &                           &          &                     &      \\ \hline%
(4-2-2) & $u_1 \ra L$, $u_3, 2u_1 \nra L$ over $\Z_2$       & $\Frv L \subset 2^3\cdot 5^2\Fro$                                          %
                                                                                        & $1,209$  &$\qf{1, 200}$        & 44   \\       %
        & $u_{+} \ra L$, $u_{-}, 5u_{-} \nra L$ over $\Z_5$ &                           &          &                     &      \\ \hline%
(5-1)   & $u_1, 2u_1 \ra L$, $u_3 \nra L$ over $\Z_2$       & \multicolumn{4}{c}{N.A.}                                          \\       %
        & $u_{+} \nra L$, $u_{-} \ra L$ over $\Z_5$         & \multicolumn{4}{c}{ }                                             \\\hline %
(5-2-1) & $u_1 \ra L$, $u_3, 2u_1 \nra L$ over $\Z_2$       & $\Frv L \subset 2^3 \cdot 5 \Fro$                                                     %
                                                                                        & $5,13$   & $\qf{5,8}$          & 12   \\       %
        & $u_{+} \nra L$, $u_{-}, 5u_{+} \ra L$ over $\Z_5$ &                           &          &                     &      \\ \hline%
(5-2-2) & $u_1 \ra L$, $u_3, 2u_1 \nra L$ over $\Z_2$       & $\Frv L \subset 2^3 \cdot 5^2 \Fro$                                                   %
                                                                                        & $13,17$  & $[12,1+2\omega,17]$ & 8    \\       %
        & $u_{+}, 5u_{+} \nra L$, $u_{-} \ra L$ over $\Z_5$ &                           &          &                     &      \\ \hline%
(6-1)   & $u_1 \nra L$, $u_3, 2u_3 \ra L$ over $\Z_2$       & $\Frv L \subset 2^2\Fro$  & $2,3$    & $[2,-1+\omega,3] \perp \qf{4}$      %
                                                                                                                         & none \\       %
        & $u_{+}, u_{-} \ra L$ over $\Z_5$                  &                           &          &                     &      \\ \hline%
(6-2)   & $u_1, 2u_3 \nra L$, $u_3 \ra L$ over $\Z_2$       & $\Frv L \subset 2^3\Fro$  & $2,3$    & $[2,-1+\omega,3] \perp \qf{8}$      %
                                                                                                                         & 8    \\       %
        & $u_{+}, u_{-} \ra L$ over $\Z_5$                  &                           &          &                     &      \\ \hline%
(7-1)   & $u_1 \nra L$, $u_3, 2u_3 \ra L$ over $\Z_2$       & $\Frv L \subset 2^2 \cdot 5 \Fro$                                                     %
                                                                                        & $4,6$    & \multirow{3}{*}{\hspace{-1.2ex}\tiny$\terlattice{4}{-2+2\omega}{-2}{}{6}{1+\omega}{}{}{11}$} %
                                                                                                                         & 10   \\       %
        & $u_{+} \ra L$, $u_{-} \nra L$ over $\Z_5$         &                           &          &                           \\
        &                                                   &                           &          & \\ \hline%
(7-2)   & $u_1, 2u_3 \nra L$, $u_3 \ra L$ over $\Z_2$       & \multicolumn{4}{c}{N.A.}                                          \\       %
        & $u_{+} \ra L$, $u_{-} \nra L$ over $\Z_5$         & \multicolumn{4}{c}{ }                                             \\ \hline%
(8-1)   & $u_1 \nra L$, $u_3, 2u_3 \ra L$ over $\Z_2$       & $\Frv L \subset 2^2 \cdot 5 \Fro$                                                     %
                                                                                        & $2,3$    & $[2, -1+\omega,3]\perp\qf{20}$      %
                                                                                                                         & none \\       %
        & $u_{+} \nra L$, $u_{-} \ra L$ over $\Z_5$         &                           &          &                     &      \\ \hline%
(8-2)   & $u_1, 2u_3 \nra L$, $u_3 \ra L$ over $\Z_2$       & \multicolumn{4}{c}{N.A.}                                          \\       %
        & $u_{+} \nra L$, $u_{-} \ra L$ over $\Z_5$         & \multicolumn{4}{l}{}                                              \\ \hline%

\end{tabular}
\caption{escalation when $m=5$} %
\label{tbl:m=5}%%
\end{table}
\end{footnotesize}

For the case (1), after second escalation, $L \cong \qf{1,5}$ or $L$ contains a lattice
$\qf{1,10}$. If $L$ contains $\qf{1,10}$, then $L$ contains a binary lattice $\qf{1} \perp
\binlattice{10}{\alpha}{\conj\alpha}{5\beta}$ with $50\beta - \alpha{\conj\alpha} = 0$ and
$\alpha$, $\beta \in \Fro$. Thus $\alpha = -5+5\omega$ and $5\beta = 15$. We have candidates
$\qf{1, 10}$ and $\qf{1} \perp \binlattice{10}{-5+5\omega}{-5+5\comega}{15}$.

For the case (2-1), after second escalation, $L \cong \binlattice2113$ or $L$ contains a unary
unimodular lattice $\binlattice{2}{-1+\omega}{-1+\comega}{3}$ which splits $L$. Since $5 \ra L$ and
$5 \nra \binlattice{2}{-1+\omega}{-1+\comega}{3}$, we can get a candidate
$\binlattice{2}{-1+\omega}{-1+\comega}{3} \perp \qf{5}$ by comparing volume of $L$. Similarly, we
can get candidates for the cases (2-2), (6-1), (6-2) and (8-1).

For the cases (4-1), (5-1), (7-2) and (8-2), $u_{+}, u_{-} \ra \lambda_{5^k}(L)$ over $\Z_5$ for
some $k$ by Lemma \ref{Lem:local_universality_of_Watson_transformation}. Since $\lambda_{5^k}(L)$
cannot be regular by the case (3-1) or (6-2), $L$ cannot be regular.

For the case (7-1), after second escalation, $L \cong \binlattice{4}{2}{2}{6}$ or $L$ contains a
unary sublattice $2\binlattice2{-1+\omega}{-1+\comega}3$. The first lattice is not primitive and it
is not in our consideration. In the second case, since $11 \ra L$ and $11 \nra
2\binlattice2{-1+\omega}{-1+\comega}3$, we conclude that $L \cong
\terlattice{4}{-2+2\omega}{-2}{-2+2\comega}{6}{1+\omega}{-2}{1+\comega}{11}$ with $\Frv L =20\Fro$.

%%%%%%%%%%%%%%%%%%%%%%%%%%%%%%%%%%%%%%%%%%%%%%%%%%%%%%%%%%%%%%%%%%%%%%%%%%%%%%
\Case{II [$m=6$]} Note that
\begin{gather*}
  u_1 \ra L \Longleftrightarrow u_7 \ra L
  \Longrightarrow 2u_3, 2u_5 \ra L
  \text{ over } \Z_2; \\
  u_3 \ra L \Longleftrightarrow u_5 \ra L
  \Longrightarrow 2u_1, 2u_7 \ra L
  \text{ over } \Z_2; \\
  u_{+} \ra L \Longrightarrow 3u_{-} \ra L \text{ over } \Z_3;\quad
  u_{-} \ra L \Longrightarrow 3u_{+} \ra L \text{ over } \Z_3.
\end{gather*}

According to the representability of $u_1, u_3$ over $\Z_2$ and $u_{+}, u_{-}$ over $\Z_3$, we have
eight cases (See Table \ref{tbl:m=6}). Since the ring $\Fro$ of integers is not a PID, we should
also consider nonfree lattices.

\begin{footnotesize}
\begin{table}[!h]
\centering

\begin{tabular}{p{0.9cm}p{4cm}|p{1.6cm}|p{0.8cm}|p{2.5cm}|p{0.6cm}} \hline%
\multicolumn{2}{c|}{Local Condition}                & \multicolumn{1}{c|}{Volume} & Ess.\#  & \multicolumn{1}{c|}{Reduced Lattice} & Exc.\# \\ \hline%

(1-1) & $u_1, u_3 \ra L$ over $\Z_2$                      & $\Frv L \subset 3\Fro$    & $1,3 $    & $\qf{1,3}$          & none \\       %
      & $u_{+}, 3u_{+} \ra L$, $u_{-} \nra L$ over $\Z_3$ &                           &           &                     &      \\ \hline%
(1-2) & $u_1, u_3 \ra L$ over $\Z_2$                      & $\Frv L \subset 3^2\Fro$  & $1,13 $   & $\qf{1,9}$          & 27   \\       %
      & $u_{+} \ra L$, $u_{-}, 3u_{+} \nra L$ over $\Z_3$ &                           &           &                     &      \\ \hline%
(2-1) & $u_1, u_3 \ra L$ over $\Z_2$                      & $\Frv L \subset 3^2\Fro$  & $2,3 $    & $[2,\omega,3]\perp\qf{9}$ & 26 \\   %
      & $u_{+}, 3u_{-} \nra L$, $u_{-} \ra L$ over $\Z_3$ &                           &           &                     &      \\ \hline%
(2-2) & $u_1, u_3 \ra L$ over $\Z_2$                      & $\Frv L \subset 3\Fro$    & $2,3 $    & $[2,0,3]$           & 6    \\       %
      & $u_{+} \nra L$, $u_{-}, 3u_{-} \ra L$ over $\Z_3$ &                           &           & $[2,\omega,3]\perp\qf{3}$ & 9  \\   %
      &                                                   &                           &           & $[2,\omega,3]\perp3[2,\omega,3]$    %
                                                                                                                        & none \\ \hline%
(3-1) & $u_1, 2u_1 \ra L$, $u_3 \nra L$ over $\Z_2$       & $\Frv L \subset 2^3\Fro$  & $1,2 $    & N.A.                &      \\       %
      & $u_{+}, u_{-} \ra L$ over $\Z_3$                  &                           &           &                     &      \\ \hline%
(3-2) & $u_1 \ra L$, $2u_1, u_3 \nra L$ over $\Z_2$       & $\Frv L \subset 2^4\Fro$  & $1,17 $   & $\qf{1,16}$         & 7    \\       %
      & $u_{+}, u_{-} \ra L$ over $\Z_3$                  &                           &           &                     &      \\ \hline%
(4)   & $u_1 \ra L$, $u_3 \nra L$ over $\Z_2$             & \multicolumn{4}{c}{N.A.}                                           \\       %
      & $u_{+} \ra L$, $u_{-} \nra L$ over $\Z_3$         & \multicolumn{4}{c}{}                                               \\ \hline%
(5)   & $u_1 \ra L$, $u_3 \nra L$  over $\Z_2$            & \multicolumn{4}{c}{N.A.}                                           \\       %
      & $u_{+} \nra L$, $u_{-} \ra L$ over $\Z_3$         & \multicolumn{4}{l}{}                                               \\ \hline%
(6)   & $u_1 \nra L$, $u_3 \ra L$ over $\Z_2$             & $\Frv L \subset 2^3\Fro$  & $2,3 $    & $[2,\omega,3]\perp\qf{8}$ & 6 \\    %
      & $u_{+}, u_{-} \ra L$ over $\Z_3$                  &                           &           &                    &       \\ \hline%
(7)   & $u_1 \nra L$, $u_3 \ra L$ over $\Z_2$             & \multicolumn{4}{c}{N.A.}                                           \\       %
      & $u_{+} \ra L$, $u_{-} \nra L$ over $\Z_3$         & \multicolumn{4}{c}{}                                               \\ \hline%
(8)   & $u_1 \nra L$, $u_3 \ra L$ over $\Z_2$             & \multicolumn{4}{c}{N.A.}                                           \\       %
      & $u_{+} \nra L$, $u_{-} \ra L$ over $\Z_3$         & \multicolumn{4}{c}{}                                               \\ \hline%
\end{tabular}
\caption{escalation when $m=6$} %
\label{tbl:m=6}%%
\end{table}
\end{footnotesize}

For the case (2-2), after second escalation, $L \cong \binlattice2003$ or $L$ contains
$\binlattice2\omega\comega3$ by the volume condition. For the second case, note that $6 \ra L$ and
$6 \nra \binlattice2\omega\comega3$. Since $\binlattice2\omega\comega3$ is a unary sublattice which
splits $L$, $L \cong \binlattice2\omega\comega3 \perp \qf{3}$ or $L$ contains a lattice
$\binlattice2\omega\comega3 \perp \qf{6}$ by the volume condition. Since $9$ is an exceptional
number of $\binlattice2\omega\comega3 \perp \qf{3}$, it is not regular.
For the last case note that $15 \ra L$ and $15 \nra \binlattice2\omega\comega3 \perp \qf{6}$. So
$L$ contains $\binlattice2{\omega}{\comega}3 \perp \binlattice6{\delta}{\conj\delta}{3\rho}$ with
$18\rho -\delta{\conj\delta} = 0$ with $\delta, \rho \in \Fro$. Then we have a candidate
$\binlattice2{\omega}{\comega}3 \perp 3\binlattice2{\omega}{\comega}3$ with volume $3\Fro$. Note
that this lattice is isometric to the \emph{free} lattice $\binlattice9{4\omega}{4\comega}{11}$.
Similarly, we can get candidates for the cases (2-1) and (6).

For the case (4), (5), (7) and (8), $u_{+}, u_{-} \ra \lambda_{3^k}(L)$ over $\Z_3$ for some $k$ by
Lemma \ref{Lem:local_universality_of_Watson_transformation}. since $\lambda_{3^k}(L)$ cannot be
regular by the case (3-1), (3-2) or (6), $L$ cannot be regular.

%%%%%%%%%%%%%%%%%%%%%%%%%%%%%%%%%%%%%%%%%%%%%%%%%%%%%%%%%%%%%%%%%%%%%%%%%%%%%%
\Case{II [$m=7$]} Note that
\begin{gather*}
  u_1 \ra L \Longleftrightarrow u_3 \ra L
  \Longleftrightarrow u_5 \ra L \Longleftrightarrow u_7 \ra L
  \Longrightarrow 2u_1, 2u_3, 2u_5, 2u_7 \ra L
  \text{ over } \Z_2; \\
  u_{+} \ra L \Longrightarrow 7u_{+} \ra L \text{ over } \Z_7;\quad
  u_{-} \ra L \Longrightarrow 7u_{-} \ra L \text{ over } \Z_7.
\end{gather*}
According to the representability of  $u_{+}, u_{-}$ over $\Z_7$, we have two cases (See Table
\ref{tbl:m=7}).

\begin{footnotesize}
\begin{table}[!h]
\centering

\begin{tabular}{p{0.9cm}p{4cm}|p{1.6cm}|p{0.8cm}|p{2.5cm}|p{0.6cm}} \hline%
\multicolumn{2}{c|}{Local Condition}                & \multicolumn{1}{c|}{Volume} & Ess.\#  & \multicolumn{1}{c|}{Reduced Lattice} & Exc.\# \\ \hline%

(1)   & $u_{+} \ra L$, $u_{-} \nra L$ over $\Z_7$   & $\Frv L \subset 7 \Fro$   & $1,15$   & $\qf{1,7}$          & none   \\       %
      &                                             &                           &          & $\qf{1,14}$         & none   \\ \hline%
(2)   & $u_{+} \nra L$, $u_{-} \ra L$ over $\Z_7$   & $\Frv L \subset 7 \Fro$   & $3, 5$   & $[3,1,5]$           & 7      \\%
      &                                             &                           &          & $[3,\omega,3]$      & none   \\ \hline%
\end{tabular}
\caption{escalation when $m=7$} %
\label{tbl:m=7}%%
\end{table}
\end{footnotesize}

%%%%%%%%%%%%%%%%%%%%%%%%%%%%%%%%%%%%%%%%%%%%%%%%%%%%%%%%%%%%%%%%%%%%%%%%%%%%%%
\Case{II [$m=10$]} Note that
\begin{gather*}
  u_1 \ra L \Longleftrightarrow u_3 \ra L
  \Longrightarrow 2u_5, 2u_7 \ra L
  \text{ over } \Z_2; \\
  u_5 \ra L \Longleftrightarrow u_7 \ra L
  \Longrightarrow 2u_1, 2u_3 \ra L
  \text{ over } \Z_2; \\
  u_{+} \ra L \Longrightarrow 5u_{-} \ra L \text{ over } \Z_5;\quad
  u_{-} \ra L \Longrightarrow 5u_{+} \ra L \text{ over } \Z_5.
\end{gather*}
According to the representability of $u_1, u_5$ over $\Z_2$ and
 $u_{+}, u_{-}$ over $\Z_7$, we have eight
cases (See Table \ref{tbl:m=10}). Since the ring $\Fro$ of integers is not a PID, we should also
consider nonfree lattices.

\begin{footnotesize}
\begin{table}[!h]
\centering

\begin{tabular}{p{0.9cm}p{4cm}|p{1.6cm}|p{0.8cm}|p{2.5cm}|p{0.6cm}} \hline%
\multicolumn{2}{c|}{Local Condition}                & \multicolumn{1}{c|}{Volume} & Ess.\#  & \multicolumn{1}{c|}{Reduced Lattice} & Exc.\# \\ \hline%

(1)   & $u_1, u_5 \ra L$ over $\Z_2$                & $\Frv L \subset 5\Fro$    & $1,6 $   & $\qf{1,5}$          & none \\       %
      & $u_{+} \ra L$, $u_{-} \nra L$ over $\Z_5$   &                           &          &                     &      \\ \hline%
(2)   & $u_1, u_5 \ra L$ over $\Z_2$                & $\Frv L \subset 5\Fro$    & $2, 3$   & $[2,1,3]$           & 5    \\       %
      & $u_{+} \nra L$, $u_{-} \ra L$ over $\Z_5$   &                           &          &                     &      \\ \hline%
(3)   & $u_1 \ra L$, $u_5 \nra L$ over $\Z_2$       & $\Frv L \subset 2^3 \Fro$ & $1, 3$   & N.A.                &      \\       %
      & $u_{+}, u_{-} \ra L$ over $\Z_5$            &                           &          &                     &      \\ \hline%
(4)   & $u_1 \ra L$, $u_5 \nra L$ over $\Z_2$       & \multicolumn{4}{c}{N.A.}                                          \\       %
      &  $u_{+} \ra L$, $u_{-} \nra L$ over $\Z_5$  & \multicolumn{4}{c}{}                                              \\ \hline%
(5)   & $u_1 \ra L$, $u_5 \nra L$ over $\Z_2$       & \multicolumn{4}{c}{N.A.}                                          \\       %
      & $u_{+} \nra L$, $u_{-} \ra L$ over $\Z_5$   & \multicolumn{4}{c}{}                                              \\ \hline%
(6)   & $u_1 \nra L$, $u_5 \ra L$ over $\Z_2$       & $\Frv L \subset 2^3 \Fro$ & $2, 5$   & N.A.                &      \\       %
      & $u_{+}, u_{-} \ra L$ over $\Z_5$            &                           &          &                     &      \\ \hline%
(7)   & $u_1 \nra L$, $u_5 \ra L$ over $\Z_2$       & \multicolumn{4}{c}{N.A.}                                          \\       %
      & $u_{+} \ra L$, $u_{-} \nra L$ over $\Z_5$   & \multicolumn{4}{c}{}                                              \\ \hline%
(8)   & $u_1 \nra L$, $u_5 \ra L$ over $\Z_2$       & \multicolumn{4}{c}{N.A.}                                          \\       %
      & $u_{+} \nra L$, $u_{-} \ra L$ over $\Z_5$   & \multicolumn{4}{c}{}                                              \\ \hline%
\end{tabular}
\caption{escalation when $m=10$} %
\label{tbl:m=10}%%
\end{table}
\end{footnotesize}

For the case (6), $L$ contains a unimodular sublattice $\binlattice2{\omega}{\comega}5$ which
splits $L$. Since $6 \ra L$ and $6 \nra \binlattice2{\omega}{\comega}5$, $L$ contains
$\binlattice2{\omega}{\comega}5 \perp \qf{a}$ with $a \leq 6$. But this lattice cannot have the
volume contained in $8 \Fro$. So we have no candidates.

For the case (4), (5), (7) and (8), $u_{+}, u_{-} \ra \lambda_{5^k}(L)$ over $\Z_3$ for some $k$ by
Lemma \ref{Lem:local_universality_of_Watson_transformation}. Since $\lambda_{5^k}(L)$ cannot be
regular by the case (3) or (6), $L$ cannot be regular.\\

%%%%%%%%%%%%%%%%%%%%%%%%%%%%%%%%%%%%%%%%%%%%%%%%%%%%%%%%%%%%%%%%%%%%%%%%%%%%
\Case{II [$m=11$]} Note that
\begin{gather*}
  u_1 \ra L \Longleftrightarrow u_3 \ra L
  \Longleftrightarrow u_5 \ra L \Longleftrightarrow u_7 \ra L
  \text{ over } \Z_2; \\
  u_{+} \ra L \Longrightarrow 11u_{+} \ra L \text{ over }
  \Z_{11};\quad
  u_{-} \ra L \Longrightarrow 11u_{-} \ra L \text{ over } \Z_{11}.
\end{gather*}
According to the representability of $2u_1$ over $\Z_2$ and $u_{+}, u_{-}$ over $\Z_7$, we have
five cases (See Table \ref{tbl:m=11}).

\begin{footnotesize}
\begin{table}[!h]
\centering

\begin{tabular}{p{0.9cm}p{4cm}|p{1.6cm}|p{0.8cm}|p{2.5cm}|p{0.6cm}} \hline%
\multicolumn{2}{c|}{Local Condition}                & \multicolumn{1}{c|}{Volume} & Ess.\#  & \multicolumn{1}{c|}{Reduced Lattice} & Exc.\# \\ \hline%

(1)   & $2u_1 \ra L$ over $\Z_2$                     & $\Frv L \subset 11\Fro$   & $1,14$   & $\qf{1,11}$         & none \\       %
      & $u_{+} \ra L$, $u_{-} \nra L$ over $\Z_{11}$ &                           &          &                     &      \\ \hline%
(2)   & $2u_1 \ra L$ over $\Z_2$                     & $\Frv L \subset 11\Fro$   & $2, 7$   & $[2,\omega,7]$      & 11   \\       %
      & $u_{+} \nra L$, $u_{-} \ra L$ over $\Z_{11}$ &                           &          &                     &      \\ \hline%
(3)   & $2u_1 \nra L$ over $\Z_2$                    & $\Frv L \subset 2^2\Fro$  & $1, 7$   & $\qf{1,4}$          & none \\       %
      & $u_{+}, u_{-} \ra L$ over $\Z_{11}$          &                           &          &                     &      \\ \hline%
(4)   & $2u_1 \nra L$ over $\Z_2$                    & $\Frv L \subset 2^2 \cdot 11\Fro$                                          %
                                                                                 & $1,91$   & $\qf{1,44}$         & none \\       %
      &  $u_{+} \ra L$, $u_{-} \nra L$ over $\Z_{11}$&                           &          & $\qf{1,88}$         & 77   \\ \hline%
(5)   & $2u_1 \nra L$ over $\Z_2$                    & $\Frv L \subset 2^2 \cdot 11\Fro$                                          %
                                                                                 & $7,13 $ & $[7,\omega,13]$      & 8    \\       %
      & $u_{+} \nra L$, $u_{-} \ra L$ over $\Z_{11}$ &                           &         & $[7, 2\omega, 8]$    & 11   \\ \hline%
\end{tabular}
\caption{escalation when $m=11$} %
\label{tbl:m=11}%%
\end{table}
\end{footnotesize}

%%%%%%%%%%%%%%%%%%%%%%%%%%%%%%%%%%%%%%%%%%%%%%%%%%%%%%%%%%%%%%%%%%%%%%%%%%%%%%
\Case{II [$m=15$]} Note that
\begin{gather*}
  u_1 \ra L \Longleftrightarrow u_3 \ra L
  \Longleftrightarrow u_5 \ra L \Longleftrightarrow u_7 \ra L
  \Longrightarrow 2u_1,2u_3,2u_5,2u_7 \ra L
  \text{ over } \Z_2; \\
  u_{+} \ra L \Longrightarrow 3u_{-} \ra L \text{ over } \Z_3;\quad
  u_{-} \ra L \Longrightarrow 3u_{+} \ra L \text{ over } \Z_3; \\
  u_{+} \ra L \Longrightarrow 5u_{-} \ra L \text{ over } \Z_5;\quad
  u_{-} \ra L \Longrightarrow 5u_{+} \ra L \text{ over } \Z_5.
\end{gather*}
According to the representability of $u_{+}$, $u_{-}$ over $\Z_5$ and $u_{+}$, $u_{-}$ over $\Z_7$,
we have eight cases (See Table \ref{tbl:m=15}).

\begin{footnotesize}
%\begin{small}
\begin{table}[h]% m=15
\centering
\begin{tabular}{p{0.9cm}p{4cm}|p{1.6cm}|p{0.8cm}|p{2.5cm}|p{0.6cm}} \hline%
%\begin{tabular}{cc|l|l|l|l} \hline%

\multicolumn{2}{c|}{Local Condition}                & \multicolumn{1}{c|}{Volume} & Ess.\#  & \multicolumn{1}{c|}{Reduced Lattice} & Exc.\# \\ \hline%

(1)     & $u_{+}, u_{-} \ra L$ over $\Z_3$          & $\Frv L \subset 5\Fro$      & 1, 11   & $\qf{1,5}$          & none   \\   %
        & $u_{+} \ra L$, $u_{-} \nra L$ over $\Z_5$ &                             &         & $\qf{1,10}$         & 5      \\   %
        &                                           &                             &         & $\qf{1} \perp 5[2,\omega,5]$      %
                                                                                                                  & 5 \\ \hline %
(2)     & $u_{+}, u_{-} \ra L$ over $\Z_3$          & $\Frv L \subset 5\Fro$      & 2, 3    & $[2,1,3]$           & 5 \\%
        & $u_{+} \nra L$, $u_{-} \ra L$ over $\Z_5$ &                             &         & $[2,\omega,2] \perp \qf{5}$%
                                                                                                                  & none \\ \hline%

(3)     & $u_{+} \ra L$, $u_{-} \nra L$ over $\Z_3$ & $\Frv L \subset 3\Fro$      & 1, 7    & $\qf{1,3}$          & none \\ %
        & $u_{+}, u_{-} \ra L$ over $\Z_5$          &                             &         & $\qf{1,6}$          & 3   \\ \hline%
(4-1)   & $u_{+}, 3u_{+} \ra L$, $u_{-} \nra L$ over $\Z_3$%
                                                    & $\Frv L \subset 3 \cdot 5 \Fro$%
                                                                                  & 1, 21   & $\qf{1,15}$         & 45 \\ %
        & $u_{+} \ra L$, $u_{-} \nra L$ over $\Z_5$ &                             &         &                     &    \\ \hline%

(4-2)   & $u_{+} \ra L$, $u_{-}, 3u_{+} \nra L$ over $\Z_3$%
                                                    & $\Frv L \subset 3^2 \cdot 5 \Fro$%
                                                                                  & 1, 91   & $\qf{1,45}$         & 17 \\ %
        & $u_{+} \ra L$, $u_{-} \nra L$ over $\Z_5$ &                             &         & $\qf{1,90}$         & 145 \\ \hline%

(5-1-1) & $u_{+}, 3u_{+} \ra L$, $u_{-} \nra L$ over $\Z_3$%
                                                    & $\Frv L \subset 3 \cdot 5\Fro$%
                                                                                  & 3, 7    & $[3, 1+\omega, 7]$  & 15 \\ %
        & $u_{+} \nra L$, $u_{-}, 5u_{-} \ra L$ over $\Z_5$%
                                                    &                             &         &                     &      \\ \hline%
(5-1-2) & $u_{+}, 3u_{+} \ra L$, $u_{-} \nra L$ over $\Z_3$%
                                                    & $\Frv L \subset 3 \cdot 5^2\Fro$%
                                                                                  & 3, 7    & N.A.                & \\ %
        & $u_{+}, 5u_{-} \nra L$, $u_{-} \ra L$ over $\Z_5$%
                                                    &                             &         &                     &      \\ \hline%
(5-2-1) & $u_{+} \ra L$, $u_{-}, 3u_{+} \nra L$ over $\Z_3$
                                                    & $\Frv L \subset 3^2 \cdot 5\Fro$
                                                                                  & 7, 11   & $[7, 2, 7]$         & 13 \\ %
        & $u_{+} \nra L$, $u_{-}, 5u_{-} \ra L$ over $\Z_5$
                                                    &                             &         &                     &      \\ \hline%
(5-2-2) & $u_{+} \ra L$, $u_{-}, 3u_{+} \nra L$ over $\Z_3$
                                                    & $\Frv L \subset 3^2 \cdot 5^2\Fro$
                                                                                  & 7, 13   & N.A.                & \\ %
        & $u_{+}, 5u_{-} \nra L$, $u_{-} \ra L$ over $\Z_5$
                                                    &                             &         &                     &      \\ \hline%
(6-1)   & $u_{+} \nra L$, $u_{-}, 3u_{-} \ra L$ over $\Z_3$
                                                    & $\Frv L \subset 3\Fro$      & 2, 6    & $[2,\omega,2] \perp \qf{6}$
                                                                                                                  & 15 \\ %
        & $u_{+}, u_{-} \ra L$ over $\Z_5$          &                             &         & $[2,\omega,2] \perp 3[2,\omega,2]$
                                                                                                                  & none \\ %
        &                                           &                             &         & $\cong [8,{-1+4\omega},8]$
                                                                                                                  &  \\ \hline%
(6-2)   & $u_{+}, 3u_{-} \nra L$, $u_{-} \ra L$ over $\Z_3$
                                                    & $\Frv L \subset 3^2\Fro$    & 2, 3    & $[2,\omega,2] \perp \qf{9}$
                                                                                                                  & none \\ %
        & $u_{+}, u_{-} \ra L$ over $\Z_5$          &                             &         &                     &  \\ \hline%
(7-1-1) & $u_{+} \nra L$, $u_{-}, 3u_{-} \ra L$ over $\Z_3$
                                                    & $\Frv L \subset 3 \cdot 5\Fro$
                                                                                  & 5, 6    & $[5, -1+2\omega, 6]$& 9 \\ %
        & $u_{+}, 5u_{+} \ra L$, $u_{-} \nra L$ over $\Z_5$
                                                    &                             &         & $\qf{5} \perp 3[2,\omega,2]$
                                                                                                                  & 21   \\ \hline%
(7-1-2) & $u_{+} \nra L$, $u_{-}, 3u_{-} \ra L$ over $\Z_3$
                                                    & $\Frv L \subset 3 \cdot 5^2\Fro$
                                                                                  & 6, 11   & N.A.                &  \\ %
        & $u_{+} \ra L$, $u_{-}, 5u_{+} \nra L$ over $\Z_5$
                                                    &                             &         &                     &      \\ \hline%
(7-2-1) & $u_{+}, 3u_{-} \nra L$, $u_{-} \ra L$ over $\Z_3$
                                                    & $\Frv L \subset 3^2 \cdot 5\Fro$
                                                                                  & 5, 11   & $[5, 2+\omega, 11]$ & 9 \\ %
        & $u_{+}, 5u_{+} \ra L$, $u_{-} \nra L$ over $\Z_5$
                                                    &                             &         &                     &      \\ \hline%
(7-2-2) & $u_{+}, 3u_{-} \nra L$, $u_{-} \ra L$ over $\Z_3$
                                                    & $\Frv L \subset 3^2 \cdot 5^2\Fro$
                                                                                  & 11, 14  & N.A.                 & \\ %
        & $u_{+} \ra L$, $u_{-}, 5u_{+} \nra L$ over $\Z_5$
                                                    &                             &         &                      &      \\ \hline%
(8-1)   & $u_{+} \nra L$, $u_{-}, 3u_{-} \ra L$ over $\Z_3$
                                                    & $\Frv L \subset 3 \cdot 5\Fro$
                                                                                  & 2, 3    & N.A.                 & \\ %
        & $u_{+} \nra L$, $u_{-} \ra L$ over $\Z_5$ &                             &         &                      &      \\ \hline%
(8-2)   & $u_{+}, 3u_{-} \nra L$, $u_{-} \ra L$ over $\Z_3$
                                                    & $\Frv L \subset 3^2 \cdot 5\Fro$
                                                                                  & 2, 3    & $[2,\omega,2] \perp \qf{45}$
                                                                                                                   & 35 \\ %
        & $u_{+} \nra L$, $u_{-} \ra L$ over $\Z_5$ &                             &         &                      & \\ \hline%
\end{tabular}
\caption{escalation when $m=15$} %
\label{tbl:m=15}%%
\end{table}
\end{footnotesize}
%\end{small}

For the case (1), $L \cong \qf{1,5}$ or $L$ contains $\qf{1, 10}$. Since $5$ is an exceptional
number of $\qf{1,10}$, $L$ contains a lattice $\qf{1} \perp
\binlattice{10}{\alpha}{\conj\alpha}{5\beta}$ with $50\beta - \alpha{\conj\alpha} =0$ and $\alpha$,
$\beta \in \Fro$. Thus $L \cong \qf{1} \perp 5\binlattice{2}{\omega}{\comega}{2}$.

For the case (2), after second escalation, $L \cong \binlattice2113$ or $L$ contains a unary
unimodular lattice $\binlattice{2}{\omega}{\comega}{2}$ which splits $L$. Since $7 \ra L$ and $7
\nra \binlattice{2}{-1+\omega}{-1+\comega}{3}$, we can get candidate
$\binlattice{2}{\omega}{\comega}{2} \perp \qf{5}$ by comparing volume of $L$. Similarly, we can get
results for the cases (6-2), (8-1) and (8-2).

For the case (3), $L \cong \qf{1,3}$ or $L$ contains $\qf{1,6}$. When $L$ contains $\qf{1,6}$,
since $3 \ra \gen\qf{1,6}$, $3 \ra L$ and hence we have a candidate $\qf{1,3}$ by the volume
condition. Similarly, for the case (4-2), $L$ contains $\qf{1,90}$ or $L \cong \qf{1, 45}$. If $L$
contains $\qf{1,90}$, then since $45 \ra \gen\qf{1,90}$ and $45 \nra \qf{1,90}$, $L \cong
\qf{1,45}$ by the volume condition.

For the case (6-1), after second escalation, $L$ contains $\qf{2,6}$,
$\binlattice{2}{1+\omega}{1+\comega}{3}$ or $\binlattice2112$. If $L$ contains  $\qf{2,6}$, from
the condition $3 \ra L$, $L$ contains a lattice
$\terlattice20\beta06\gamma{\conj\beta}{\conj\gamma}3$ with $3 \mid (6-\beta{\conj\beta})$, $3 \mid
(18-\gamma{\conj\gamma})$ and its determinant $36-6\beta{\conj\beta}-2\gamma{\conj\gamma} = 0$. So
$L$ contains a lattice $\binlattice2{1+\omega}{1+\comega}3 \perp \qf{6} \cong
\binlattice2\omega\comega2 \perp \qf{6}$. Since $15\ra L$ and $15 \nra \binlattice2\omega\comega2
\perp \qf{6}$, we have a candidate $L \cong \binlattice2\omega\comega2 \perp 3
\binlattice2\omega\comega2$, via similar way. Note that $L$ is isometric to the binary \emph{free}
lattice $\binlattice8{-3+4\omega}{-3+4\comega}8$. If $L$ contains
$\binlattice{2}{1+\omega}{1+\comega}{3}$ or $\binlattice2112$, then we know that there is no
candidates via similar way.

%%%%%%%%%%%%%%%%%%%%%%%%%%%%%%%%%%%%%%%%%%%%%%%%%%%%%%%%%%%%%%%%%%%%%%%%%%%%%%
\Case{II [$m=19$]} Note that
\begin{gather*}
  u_1 \ra L \Longleftrightarrow u_3 \ra L
  \Longleftrightarrow u_5 \ra L \Longleftrightarrow u_7 \ra L
  \text{ over } \Z_2; \\
  u_{+} \ra L \Longrightarrow 19u_{+} \ra L \text{ over }
  \Z_{19};\quad
  u_{-} \ra L \Longrightarrow 19u_{-} \ra L \text{ over } \Z_{19}.
\end{gather*}
According to the representability of $2u_1$ over $\Z_2$ and $u_{+}, u_{-}$ over $\Z_{19}$, we have
two cases (See Table \ref{tbl:m=19}).

\begin{footnotesize}
\begin{table}[h]% m=19
\centering

\begin{tabular}{p{0.9cm}p{4cm}|p{1.6cm}|p{0.8cm}|p{2.5cm}|p{0.6cm}} \hline%
\multicolumn{2}{c|}{Local Condition}                & \multicolumn{1}{c|}{Volume} & Ess.\#  & \multicolumn{1}{c|}{Reduced Lattice} & Exc.\# \\ \hline%

(1-1)   & $2u_1 \ra L$ over $\Z_2$                     & $\Frv L \subset 19 \Fro$    & 1, 6    & N.A.                &        \\%
        & $u_{+} \ra L$, $u_{-} \nra L$ over $\Z_{19}$ &                             &         &                     & \\ \hline%

(1-2)   & $2u_1 \nra L$ over $\Z_2$                    & $\Frv L \subset 2^2 \cdot 19\Fro$ %
                                                                                     & 1, 39   & N.A.                &        \\%
        & $u_{+} \ra L$, $u_{-} \nra L$ over $\Z_{19}$ &                             &         &                     & \\ \hline%

(2-1)   & $2u_1 \ra L$ over $\Z_2$                     & $\Frv L \subset 19 \Fro$    & 2, 3    & N.A.                & \\
        & $u_{+} \nra L$, $u_{-} \ra L$ over $\Z_{19}$ &                             &         &                     & \\ \hline%

(2-2)   & $2u_1 \nra L$ over $\Z_2$                    & $\Frv L \subset 2^2 \cdot 19\Fro$ %
                                                                                     & 3, 13   & N.A.                & \\
        & $u_{+} \nra L$, $u_{-} \ra L$ over $\Z_{19}$ &                             &         &                     & \\ \hline%
\end{tabular}
\caption{escalation when $m=19$} %
\label{tbl:m=19}%%
\end{table}
\end{footnotesize}

%%%%%%%%%%%%%%%%%%%%%%%%%%%%%%%%%%%%%%%%%%%%%%%%%%%%%%%%%%%%%%%%%%%%%%%%%%%%%%
\Case{II [$m=23$]} Note that
\begin{gather*}
  u_1 \ra L \Longleftrightarrow u_3 \ra L
  \Longleftrightarrow u_5 \ra L \Longleftrightarrow u_7 \ra L
  \Longrightarrow 2u_1,2u_3,2u_5,2u_7 \ra L
  \text{ over } \Z_2; \\
  u_{+} \ra L \Longrightarrow 23u_{+} \ra L \text{ over }
  \Z_{23};\quad
  u_{-} \ra L \Longrightarrow 23u_{-} \ra L \text{ over } \Z_{23}.
\end{gather*}
According to the representability of  $u_{+}, u_{-}$ over $\Z_{23}$, we have two cases (See Table
\ref{tbl:m=23}).

\begin{footnotesize}
\begin{table}[h]% m=23
\centering

\begin{tabular}{p{0.9cm}p{4cm}|p{1.6cm}|p{0.8cm}|p{2.5cm}|p{0.6cm}} \hline%
\multicolumn{2}{c|}{Local Condition}                & \multicolumn{1}{c|}{Volume} & Ess.\#  & \multicolumn{1}{c|}{Reduced Lattice} & Exc.\# \\ \hline%

(1)     & $u_{+} \ra L$, $u_{-} \nra L$ over $\Z_{23}$ & $\Frv L \subset 23\Fro$     & 1, 2    & N.A.                &        \\ \hline%
(2)     & $u_{+} \nra L$, $u_{-} \ra L$ over $\Z_{23}$ & $\Frv L \subset 23\Fro$     & 5, 7    & $[5,2+\omega,7]$    & 10     \\ \hline%
\end{tabular}
\caption{escalation when $m=23$} %
\label{tbl:m=23}%%
\end{table}

\end{footnotesize}

%%%%%%%%%%%%%%%%%%%%%%%%%%%%%%%%%%%%%%%%%%%%%%%%%%%%%%%%%%%%%%%%%%%%%%%%%%%%%%
\Case{II [$m=31$]} Note that
\begin{gather*}
  u_1 \ra L \Longleftrightarrow u_3 \ra L
  \Longleftrightarrow u_5 \ra L \Longleftrightarrow u_7 \ra L
  \Longrightarrow 2u_1,2u_3,2u_5,2u_7 \ra L
  \text{ over } \Z_2; \\
  u_{+} \ra L \Longrightarrow 31u_{+} \ra L \text{ over } \Z_{31};\quad
  u_{-} \ra L \Longrightarrow 31u_{-} \ra L \text{ over } \Z_{31}.
\end{gather*}
According to the representability of $u_{+}, u_{-}$ over $\Z_{31}$, we have two cases (See Table
\ref{tbl:m=31}).

\begin{footnotesize}
\begin{table}[h] % m=31
\centering

\begin{tabular}{p{0.9cm}p{4cm}|p{1.6cm}|p{0.8cm}|p{2.5cm}|p{0.6cm}} \hline%
\multicolumn{2}{c|}{Local Condition}                & \multicolumn{1}{c|}{Volume} & Ess.\#  & \multicolumn{1}{c|}{Reduced Lattice} & Exc.\# \\ \hline%

(1)     & $u_{+} \ra L$, $u_{-} \nra L$ over $\Z_{31}$ & $\Frv L \subset 31\Fro$     & 1, 2    & N.A.                &        \\ \hline%
(2)     & $u_{+} \nra L$, $u_{-} \ra L$ over $\Z_{31}$ & $\Frv L \subset 31\Fro$     & 3, 6    & N.A.                &        \\ \hline%
\end{tabular}
\caption{escalation when $m=31$} %
\label{tbl:m=31}%%
\end{table}
\end{footnotesize}

\section{Complete list of binary regular Hermitian lattices}\label{List}

In this section, we will prove the regularity of each candidate. If its class number is one, then
it is trivially regular. We know that the class numbers of following lattices are one and hence
they are regular. Some diagonal lattices are checked by \cite{gO-71}, and the other nondiagonal
lattices are checked by authors via comparing discriminants and local properties.\\

\begin{center}
\begin{footnotesize}
\begin{tabular}{c|l} \hline%
field    & class number one lattices                   \\ \hline%

$\Q{-1}$ & $\qf{1,4}$,                                         %
           $\binlattice{2}{-1+\omega}{-1+\comega}{3}$,         %
           $\binlattice{3}{-1+\omega}{-1+\comega}{6}$,         %
           $\binlattice3113$                          \\ \hline%

$\Q{-3}$ & $\qf{1,3}$, $\qf{1,4}$, $\qf{1,6}$, $\qf{2,3}$,     %
           $\binlattice2112$,                                  %
           $\binlattice2115$,                                  %
           $\binlattice{3}{1+\omega}{1+\comega}{5}$,           %
           $\binlattice5228$                          \\ \hline%

$\Q{-7}$ & $\binlattice{3}{\omega}{\comega}{3}$       \\ \hline%
\end{tabular}
\end{footnotesize}
\end{center}

Now, we prove all the surviving candidates are actually regular. Most of proofs are using a ternary
regular sublattice whose class number one. Whereas in the the proof of Case II [$m=7$] (2), we give
an efficient bound for the numbers represented by a specific quaternary quadratic form. This new
arithmetic method uses ternary quadratic forms which are not regular.\\

\Case{I [$m=1$]}

(1) $L = \qf{1,8}$ is regular over $\Q{-1}$.
\begin{proof}
Note that
    \[
    H(\gen L) = \{ n \in \N_0 \mid n \equiv 0, 1 \pmod{4} \text{ or } n \equiv 2 \pmod{8} \}.
    \]
Since $\qf{2,8} = 2\qf{1,4}$ is a sublattice of $L$, if $n \ra \qf{1,4}$, then $2n \ra \qf{2,8}$
and hence $2n \ra L$. From the Case I [$m=1$] (1), we know that $\qf{1,4}$ is regular, which
represents all positive integers $n$ such that $n \equiv 0, 1, 2 \pmod{4}$. Hence $L$ represents
all positive integers $n$ such that $n \equiv 0 \pmod{4}$ or $n \equiv 2 \pmod{8}$. On the other
hand, the associated quadratic lattice of $L$ is
    \[
    x\conj{x} + 8y\conj{y} = x_1^2 + x_2^2 + 8y_1^2 + 8y_2^2.
    \]
Since it has a regular sublattice $x_1^2 + x_2^2 + 8y_1^2$ \cite{wJ-iK-aS-97}, which represents all
positive integers $n \equiv 1 \pmod{4}$ , $L$ represents all positive integers $n \equiv 1
\pmod{4}$. Therefore $L$ is regular.
\end{proof}

(2) $L = \qf{1,16}$ is regular over $\Q{-1}$.
\begin{proof}
Note that
    \[
    H(\gen L) = \{ n \in \N_0 \mid n \equiv 1 \pmod{4}, n \equiv 0, 2 \pmod{8} \text{ or } n \equiv 4 \pmod{16} \}.
    \]
Since $\qf{2,16} = 2\qf{1,8}$ is a sublattice of $L$, if $n \ra \qf{1,8}$, then $2n \ra \qf{2,16}$
and hence $2n \ra L$. From the Case I [$m=1$] (2), we know that $\qf{1,8}$ is regular, which
represents all positive integers $n$ such that $n \equiv 0, 1 \pmod{4}$ or $n \equiv 2 \pmod{8}$.
Hence $L$ represents all positive integers $n$ such that $n \equiv 0, 2 \pmod{4}$ or $n \equiv 4
\pmod{16}$. On the other hand, the associated quadratic lattice of $L$ is
    \[
    x\conj{x} + 16y\conj{y} = x_1^2 + x_2^2 + 16y_1^2 + 16y_2^2.
    \]
Since it has a regular sublattice $x_1^2 + x_2^2 + 16y_1^2$ \cite{wJ-iK-aS-97}, which represents
all positive integers $n \equiv 1 \pmod{4}$, $L$ represents all positive integers $n \equiv 1
\pmod{4}$. Therefore $L$ is regular.
\end{proof}

\Case{I [$m=2$]}

(1) $L = \qf{1,8}$ is regular over $\Q{-2}$.
\begin{proof}
Note that
    \[
    H(\gen L) = \{ n \in \N_0 \mid n \equiv 1,3 \pmod{8} \text{ or } n \equiv 0 \pmod{2} \}.
    \]
Since $\qf{2,8} = 2\qf{1,4}$ is a sublattice of $L$ and $\qf{1,4}$ is universal
\cite{agE-aK-97(1)}, $n \ra L$ for all $n \equiv 0 \pmod{2}$. On the other hand, the associated
quadratic form of $L$ is
    \[
    x\conj{x} + 8y\conj{y} = x_1^2 + 2 x_2^2 + 8 y_1^2 + 16 y_2^2.
    \]
Since it has a regular sublattice $x_1^2 + 2 x_2^2 + 8 y_1^2$ \cite{wJ-iK-aS-97}, which represents
all positive integers $n \equiv 1, 3 \pmod{8}$, these $n$ are all represented by $L$. Therefore $L$
is regular.
\end{proof}

(2) $L = \qf{1,16}$ is regular over $\Q{-2}$.
\begin{proof}
Note that
    \[
    H(\gen L) = \{ n \in \N_0 \mid n \equiv 0 \pmod{4}, n \equiv 1,3 \pmod{8} \text{ or } n \equiv 2, 6 \pmod{16} \}.
    \]
Since $\qf{4,16} = 4\qf{1,4}$ is a sublattice of $L$ and $\qf{1, 4}$ is universal
\cite{agE-aK-97(1)}, $n \ra L$ for all $n \equiv 0 \pmod{4}$. On the other hand, the associated
quadratic form of $L$ is
    \[
    x\conj{x} + 8y\conj{y} = x_1^2 + 2 x_2^2 + 16 y_1^2 + 32 y_2^2.
    \]
Since it has a regular sublattice $x_1^2 + 2 x_2^2 + 16 y_1^2$ \cite{wJ-iK-aS-97}, which represents
all positive integers $n \equiv 1, 3 \pmod{8}$ or $n \equiv 2, 6 \pmod{16}$, these $n$ are all
represented by $L$. Therefore $L$ is regular.
\end{proof}

(3) $L = \qf{1,32}$ is regular over $\Q{-2}$.
\begin{proof}
Note that
    \[
    H(\gen L) = \{ n \in \N_0 \mid n \equiv 0, 1, 3 \pmod{8},
    n \equiv 2, 6 \pmod{16} \text{ or } n \equiv 4, 12 \pmod{32} \}.
    \]
Since $\qf{8,32} = 8\qf{1,4}$ is a sublattice of $L$ and $\qf{1,4}$ is universal
\cite{agE-aK-97(1)}, $n \ra L$ for all $n \equiv 0 \pmod{8}$. On the other hand, the associated
quadratic form of $L$ is
    \[
    x\conj{x} + 8y\conj{y} = x_1^2 + 2 x_2^2 + 32 y_1^2 + 64 y_2^2.
    \]
Since it has a regular sublattice $x_1^2 + 2 x_2^2 + 32 y_1^2$ \cite{wJ-iK-aS-97}, which represents
all positive integers $n \equiv 1, 3 \pmod{8}$, $n \equiv 2, 6 \pmod{16}$ or $n \equiv 4, 12
\pmod{32}$ , these $n$ are all represented by $L$. Therefore $L$ is regular.
\end{proof}

(4) $L = \binlattice2{\omega}{\comega}5$ is regular over $\Q{-2}$.

\begin{proof}
Note that
    \[
    H(\gen L) = \{ n \in \N_0 \mid n \equiv 0 \pmod{2}, \text{ or } n \equiv 5, 7 \pmod{8} \}.
    \]
Consider a universal lattice $\qf{1,4}$ with basis $\{v_1, v_2\}$  \cite{agE-aK-97(1)}. We know
that $L$ is a sublattice $\qf{1,4}$ with basis $\{ \omega v_1, v_1 + v_2 \}$. If $n \equiv 0
\pmod{2}$, then $n \ra L$, since $L$ contains a sublattice $\qf{2,8} = 2\qf{1,4}$ and $\qf{1,4}$ is
universal. If $n \equiv 5, 7 \pmod{8}$, then $n \ra \qf{1,4}$ by universality of $\qf{1,4}$ and
hence
    \[
    n = H(x v_1 + y v_2) = x\conj{x} + 4y\conj{y} = x_1^2 + 2 x_2^2 + 4y_1^2 + 8y_2^2
    \]
has a solution for some $x = x_1 + \omega x_2$, $y = y_1 + \omega y_2 \in \Fro$.  Since $n \equiv
5, 7 \pmod{8}$, $x_1 \equiv y_1 \equiv 1 \pmod{2}$. So $x_1 - y_1$ is even and multiple of
$\omega$. We deduce that
    \[
    xv_1 + yv_2 = (x_1 - y_1) v_1 + \omega(x_2 - y_2) v_1 + (y_1 + \omega y_2) (v_1 + v_2) \in L
    \]
and hence $n \ra L$. Therefore $L$ is regular.
\end{proof}

\Case{II [$m=3$]}

(1) $L = \qf{1,9}$ is regular over $\Q{-3}$.
\begin{proof}
Note that
    \[
    H(\gen L) = \{ n \in \N_0 \mid n  \equiv 1 \pmod{3} \text{ or } n \equiv 0, 3 \pmod{9} \}.
    \]
Since $\qf{3,9} = 3\qf{1,3}$ is a sublattice of $L$, if $n \ra \qf{1,3}$, then $3n \ra \qf{3,9}$
and hence $3n \ra L$. From Case II [$m=3$] (1), we know that $\qf{1,3}$ is regular, which
represents all positive integers $n$ such that $n \equiv 0, 1 \pmod{3}$. Hence $L$ represents all
positive integers $n$ such that $n \equiv 0, 3 \pmod{9}$. On the other hand, the associated
quadratic lattice of $L$ is
    \[
    x\conj{x} + 9y\conj{y} = x_1^2 + x_1 x_2 + x_2^2 + 9 y_1^2 + 9 y_1 y_2 + 9 y_2^2.
    \]
Since it has a regular sublattice $x_1^2 + x_1 x_2 + x_2^2 + 9 y_1^2$ \cite{wJ-iK-aS-97}, which
represents all positive integers $n \equiv 1 \pmod{3}$, these $n$ are all represented by $L$.
Therefore $L$ is regular.
\end{proof}

(2) $L = \qf{1,12}$ is regular over $\Q{-3}$.
\begin{proof}
Note that
    \[
    H(\gen L) = \{ n \in \N_0 \mid n \equiv 0, 1, 3, 4, 7, 9  \pmod{12} \}.
    \]
Consider regular lattices $\qf{1,3}$ and $\qf{1,4}$ (see Case II [$m=3$] (1), (2)). Since
$\qf{1,3}$ represents all positive integers $n$ such that $n \equiv 0, 1 \pmod{3}$ and $L$ contains
a sublattice $\qf{4,12} = 4\qf{1,3}$, $L$ represents all positive integers $n$ such that $n \equiv
0, 4 \pmod{12}$. Since $\qf{1,4}$ represents all positive integers $n$ such that $n \equiv 0, 1, 3
\pmod{4}$ and $L$ contains a sublattice $\qf{3,12} = 3\qf{1,4}$, $L$ represents all positive
integers $n$ such that $n \equiv 0, 1, 7 \pmod{12}$. On the other hand, the associated quadratic
lattice of $L$ is
    \[
    x\conj{x} + 12y\conj{y} = x_1^2 + x_1 x_2 + x_2^2 + 12 y_1^2 + 12 y_1 y_2 + 12 y_2^2.
    \]
Since it contains a regular lattice $x_1^2 + x_1 x_2 + x_2^2 + 12 y_1^2$ \cite{wJ-iK-aS-97}, which
represents all positive integers $n \equiv 0, 1 \pmod{3}$, these $n$ are all represented by $L$.
Hence $L$ is regular.
\end{proof}

(3) $L = \qf{1,36}$ is regular over $\Q{-3}$.
\begin{proof}
Note that
    \[
    H(\gen L) = \{ n \in \Z \mid n \equiv 0, 1,3\pmod{4} \text{ and } n
    \equiv 0, 1 \pmod{3} \}.
    \]
Consider regular lattices $\qf{1,9}$, $\qf{1,12}$ (see Case II [$m=3$] (5), (6)). Since $L$
contains sublattices $\qf{4,36}=4\qf{1,9}$ and $\qf{3,36}=3\qf{1,12}$, $L$ represents all positive
integers $n$ such that $n \equiv 0 \pmod 4$ and $n \equiv 0,1  \pmod 3$, or $n \equiv 0 \pmod 3$
and $n \equiv 1 \pmod 2$. So it suffices to show that $L$ represents all positive integers $n$ such
that $n$ is odd and $n \equiv 1 \pmod 3$. The associated quadratic lattice of $L$ is
    \[
    x\conj{x} + 36y\conj{y} = x_1^2 + x_1 x_2 + x_2^2 + 36 y_1^2 + 36 y_1 y_2 + 36 y_2^2,
    \]
and it contains a sublattice isometric to $\qf{1, 3, 36, 108}_\Z = \qf{1}_\Z \perp 3\qf{1, 12,
36}_{\Z}$. Since $\qf{1,12,36}_\Z$ is regular \cite{wJ-iK-aS-97}, $3 \qf{1,12,36}_\Z =
\qf{3,36,108}_\Z$ represents all positive integer $n \equiv 3, 12 \pmod{36}$. If $n \equiv 1
\pmod{12}$ and $n \ge 49$, then  $n - a^2 \equiv 12 \pmod{36}$ for $a=1,5,7$ and hence $L$
represents $n$. If $n \equiv 7 \pmod{12}$ and $n \ge 64$, then $n - a^2 \equiv 3 \pmod{36}$ for
$a=2,4,8$ and hence $L$ represents $n$. It is an easy work to check $n \ra L$ for $1, 7, 13, 19,
25, 31, 37, 43, 49, 55$. Therefore $L$ is regular.
\end{proof}

\Case{II [$m=5$]}

(1) $L = \qf{1,8}$ is regular over $\Q{-5}$.
\begin{proof}
Note that
    \[
    H(\gen L) = \{ n \in \N_0 \mid n \equiv 0, 1 \pmod{4} \text{ or } n \equiv 6 \pmod{8} \}.
    \]
Since $\qf{4,8} = 4\qf{1,2}$ is a sublattice of $L$ and $\qf{1,2}$ is universal \cite{Iwabuchi}, $n
\ra L$ for all $n \equiv 0 \pmod{4}$. On the other hand, the associated quadratic lattice of $L$ is
\[
    x\conj{x} + 8y\conj{y} = x_1^2 + 5 x_2^2 + 8 y_1^2 + 40 y_2^2.
\]
Since it has a regular sublattice $x_1^2 + 5 x_2^2 + 8 y_1^2$ \cite{wJ-iK-aS-97}, which represents
all positive integers $n \equiv 1 \pmod{4}$ or $n \equiv 6 \pmod{8}$, these all $n$ are represented
by $L$. Therefore $L$ is regular.
\end{proof}

(2) $L = \qf{1,10}$ is regular over $\Q{-5}$.
\begin{proof}
Note that
    \[
    H(\gen L) = \{ n \in \N_0 \mid n \equiv 0, 1, 4 \pmod{5} \}.
    \]
Since  $\qf{5,10} = 5\qf{1,2}$ is a sublattice of $L$ and $\qf{1,2}$ is universal \cite{Iwabuchi},
$n \ra L$ for all $n \equiv 0 \pmod{5}$. On the other hand, the associated quadratic lattice of $L$
is
\[
    x\conj{x} + 10y\conj{y} = x_1^2 + 5 x_2^2 + 10 y_1^2 + 50 y_2^2.
\]
Since it has a regular sublattice $x_1^2 + 5 x_2^2 + 10 y_1^2$ \cite{wJ-iK-aS-97}, which represents
all positive integers $n \equiv 1, 4 \pmod{5}$, $n \ra L$ for all positive integers $n \equiv 1, 4
\pmod{5}$. Therefore $L$ is regular.
\end{proof}

(3) $L = \qf{1} \perp 5\binlattice2{-1+\omega}{-1+\comega}{3}$ is regular over $\Q{-5}$.
\begin{proof}
Note that
    \[
    H(\gen L) = \{ n \in \N_0 \mid n \equiv 0, 1, 4 \pmod{5} \}.
    \]
Since $L$ is a sublattice of a regular lattice $\qf{1,10}$ (see Case II [$m=5$] (2)), $L$ is
regular.
\end{proof}

(4) $L = \qf{1,40}$ is regular over $\Q{-5}$.
\begin{proof}
Note that
    \[
    H(\gen L) = \{ n \in \N_0 \mid n \equiv 0, 1, 4, 5, 6 \pmod{8}\}
                \cap  \{ n \in \N_0 \mid n \equiv 0,1,4 \pmod{5} \}.
    \]
Consider regular lattices $\qf{1,8}$ and $\qf{1,10}$ (see Case II [$m=5$] (1), (2)). Since
$5\qf{1,8} = \qf{5,40}$ is a sublattice of $L$, $L$ represents all positive integers $n \equiv 0,
1, 4, 5, 6 \pmod{8}$ and $n \equiv 0 \pmod{5}$. Since $4\qf{1,10} = \qf{4,40}$ is a sublattices of
$L$, $L$ represents all positive integers $n \equiv 0, 1, 4 \pmod{5}$ and $n \equiv 0 \pmod{4}$. On
the other hand, the associated quadratic lattice of $L$ is
    \[
    x\conj{x} + 40y\conj{y} = x_1^2 + 5 x_2^2 + 40 y_1^2 + 160 y_2^2.
    \]
Since it has a regular sublattice $x_1^2 + 5 x_2^2 + 40 y_1^2$ \cite{wJ-iK-aS-97}, which represents
all positive integers $n \equiv 1, 5, 6 \pmod{8}$ and $n \equiv 1,4 \pmod{5}$, these $n$ are
represented by $L$. Therefore $L$ is regular.
\end{proof}

(5) $L = \binlattice2{-1+\omega}{-1+\comega}3 \perp \qf{4}$ is regular over $\Q{-5}$.
\begin{proof}
Note that
    \[
    H(\gen L) = \{ n \in \N_0 \mid n \equiv 0 \pmod{2} \text{ or } n \equiv 3 \pmod{4} \}.
    \]
The associated quadratic lattice of $L$ is
    \[
    2 x_1^2 + 2x_1x_2 + 3x_2^2 + 4y_1^2 + 20y_2^2.
    \]
Note that it contains regular sublattices $\ell_1 = 2 x_1^2 + 2x_1x_2 + 3x_2^2 + 4y_1^2$ and
$\ell_2 = 2 x_1^2 + 2x_1x_2 + 3x_2^2 + 20y_2^2$ \cite{wJ-iK-aS-97}. The lattice $\ell_1$ represents
all positive integers $n$ if $n \equiv 0, 2,3\pmod{4}$ and $n \ne 5^d u_{+}$. The lattice $\ell_2$
represents all positive integers $n$ if $n \equiv 0, 2, 3\pmod{4}$ and $n \ne 5^e u_{+}$. Hence $n
\ra L$ if $n \equiv 0 \pmod{2}$ or $n \equiv 3\pmod{4}$.  Therefore $L$ is regular.
\end{proof}

(6) $L = \binlattice2{-1+\omega}{-1+\comega}3 \perp \qf{5}$ is regular over $\Q{-5}$.
\begin{proof}
Note that
    \[
    H(\gen L) = \{ n \in \N_0 \mid n \equiv 0, 2, 3 \pmod{5} \}.
    \]
Since $\qf{1}\perp\binlattice2{-1+\omega}{-1+\comega}3$ is universal \cite{Iwabuchi} and
$5\left(\qf{1}\perp\binlattice2{-1+\omega}{-1+\comega}3\right)$ is a sublattice of $L$, $n \ra L$
for all $n \equiv 0 \pmod{5}$. On the other hand, the associated quadratic lattice of $L$ is
    \[
    2 x_1^2 + 2x_1x_2 + 3x_2^2 + 5y_1^2 + 25y_2^2.
    \]
Since it has a regular sublattice $2 x_1^2 + 2x_1x_2 + 3x_2^2 + 5y_1^2$ \cite{wJ-iK-aS-97}, which
represents all positive integers $n \equiv 2,3 \pmod{5}$, $n \ra L$ for all positive integers $n
\equiv 2,3 \pmod{5}$. Therefore $L$ is regular.
\end{proof}

(7) $L = \binlattice2{-1+\omega}{-1+\comega}3 \perp \qf{20}$ is regular over $\Q{-5}$.
\begin{proof}
Note that
    \[
    H(\gen L) = \{ n \in \N_0 \mid n \equiv 0, 2, 3\pmod{4} \}
                \cap \{ n \in \N_0 \mid n \equiv 0, 2,3\pmod{5}  \}.
    \]
The associated quadratic lattice $\widetilde{L}$ of $L$ is
    \[
    2 x_1^2 + 2x_1x_2 + 3x_2^2 + 20y_1^2 + 100y_2^2.
    \]
We mentioned that a regular lattice $2 x_1^2 + 2x_1x_2 + 3x_2^2 + 4y^2$ represents $n$ if $n \equiv
0, 2, 3\pmod{4}$ (see Case II [$m=5$] (5)). Since $5 \left( 2 x_1^2 + 2x_1x_2 + 3x_2^2 + 4y^2
\right)$ is a sublattice of $\widetilde{L}$, $5n \ra L$ if $5n \equiv 3\pmod{4}$ or $5n\equiv 0
\pmod{2}$. Now consider another sublattice $2 x_1^2 + 2x_1x_2 + 3x_2^2 + 20y_1^2 $ (see Case II
[$m=5$] (5)). It represents $n$, if $n \equiv 0, 2, 3\pmod{4}$ and $n \equiv 2,3\pmod{5}$. Hence
the lattice $L$ is regular.
\end{proof}

\Case{II [$m=6$]}

(1) $L = \qf{1,3}$ is regular over $\Q{-6}$.
\begin{proof}
Note that
    \[
    H(\gen L) = \{ n \in \N_0 \mid n \equiv 0, 1 \pmod{3} \}.
    \]
The associated quadratic lattice of $L$ is
\[
    x\conj{x} + 3y\conj{y} = x_1^2 + 6x_2^2 + 3y_1^2 + 18y_2^2.
\]
Since it has a regular sublattice $x_1^2 + 3y_1^2 + 6x_2^2$ \cite{wJ-iK-aS-97}, which represents
all positive integers $n \equiv 0, 1 \pmod{3}$, $L$ is regular.
\end{proof}

(2) $L = \binlattice2\omega\comega3 \perp 3\binlattice2\omega\comega3$ is regular over $\Q{-6}$.
\begin{proof}
Note that
    \[
    H(\gen L) = \{ n \in \N_0 \mid n \equiv 0, 2 \pmod{3} \}.
    \]
Since $3\left( \qf{1} \perp \binlattice2\omega\comega3 \right)$ is a sublattice of $L$ and $\qf{1}
\perp \binlattice2\omega\comega3$ is universal \cite{Iwabuchi}, $L$ represents all positive
integers $n \equiv 0 \pmod{3}$. Suppose $n \equiv 2 \pmod{3}$. Put $n = n_1 + n_2$ with $n_1 \ra
\qf{1}$ and $n_2 \ra \binlattice2\omega\comega3$. Then $n_1 \equiv 0, 1 \pmod{3}$ and $n_2 \equiv
0, 2 \pmod{3}$. We have $n_1 \equiv 0 \pmod{3}$ and $n_2 \equiv 2 \pmod{3}$ from $n \equiv 2
\pmod{3}$. Since $n_1 \ra \qf{1,6}_\Z$ and $n_1 \equiv 0 \pmod{3}$, $n_1 \ra 3 \qf{2,3}_\Z$. Note
that $3 \qf{2,3}_\Z$ is an associated quadratic form of $3\binlattice2\omega\comega3$. Therefore $n
= n_1 + n_2 \ra 3\binlattice2\omega\comega3 \perp \binlattice2\omega\comega3$ and $L$ is regular.
\end{proof}

\Case{II [$m=7$]}

(1) $L = \qf{1,7}$ is regular over $\Q{-7}$.
\begin{proof}
Note that
    \[
    H(\gen L) = \{ n \in \N_0 \mid n \equiv 0, 1, 2, 4 \pmod{7} \}.
    \]
Since $\qf{7,7} = 7\qf{1,1}$ is a sublattice of $L$ and $\qf{1,1}$ is universal
\cite{agE-aK-97(1)}, $n \ra L$ for all $n \equiv 0 \pmod{7}$. On the other hand, the associated
quadratic lattice of $L$ is
\[
    x\conj{x} + 7y\conj{y} = x_1^2 + x_1 x_2 + 2 x_2^2 + 7 y_1^2 + 7 y_1 y_2 + 14 y_2^2.
\]
Since it has a regular sublattice $x_1^2 + x_1 x_2 + 2 x_2^2 + 7 y_1^2$ \cite{wJ-iK-aS-97}, which
represents all positive integers $n \equiv 1, 2, 4 \pmod{7}$, $n \ra L$ for all positive integers
$n \equiv 1, 2, 4 \pmod{7}$. Therefore $L$ is regular.
\end{proof}

(2) $L = \qf{1,14}$ is regular over $\Q{-7}$.
\begin{proof}
Note that
    \[
    H(\gen L) = \{ n \in \N_0 \mid n \equiv 0, 1, 2, 4 \pmod{7} \}.
    \]
Since $7\qf{1,2} = \qf{7,14}$ is a sublattice of $L$ and $\qf{1,2}$ is universal
\cite{agE-aK-97(1)}, $n \ra L$ for all positive integers $n \equiv 0 \pmod{7}$. So we may assume
$n$ is not a multiple of $7$ throughout this proof.

Let $\widetilde{L}$ be the associated quadratic lattice of $L$, then
    \[
    \widetilde{L} = x\conj{x} + 14y\conj{y} = x_1^2 + x_1 x_2 + 2 x_2^2 + 14 y_1^2 + 14 y_1 y_2 + 28 y_2^2
        = \binlattice{1}{1/2}{1/2}{2}_\Z \perp\binlattice{14}{7}{7}{28}_\Z.
    \]
Note that $\qf{1,7}_\Z$ and $\qf{2, 14}_\Z$ are sublattices of $\binlattice1{1/2}{1/2}2_\Z$.

\noindent (i) $n \equiv 0 \pmod{2}$: Consider an integral ternary lattice $K = \qf{2}_\Z \perp
\binlattice2114_\Z$ whose class number is one \cite{Brandt-Intrau}. It represents all positive even
integers not divisible by 7. Note that $\widetilde{L}$ has a sublattice $\qf{2, 14}_\Z \perp
\binlattice{14}77{28}_\Z = \qf{2}_\Z \perp 7 \left( \qf{2}_\Z \perp \binlattice2114_\Z \right) =
\qf{2}_\Z \perp K^7$. Suppose $n \ge 72$. Then $n \ra \qf{2}_\Z \perp K^7$ from the following
identities.
\begin{align*}
n &\equiv 1 \pmod{7} \Rightarrow n = 14k+8 =
\begin{cases}
2 \cdot 2^2 + 7 \cdot 2k & \text{ if } 7 \nmid k \\
2 \cdot 5^2 + 7 \cdot 2(k-3) & \text{ if } 7 \mid k
\end{cases}
\\
n &\equiv 2 \pmod{7} \Rightarrow n = 14k+2 =
\begin{cases}
2 \cdot 1^2 + 7 \cdot 2k & \text{ if } 7 \nmid k \\
2 \cdot 6^2 + 7 \cdot 2(k-5) & \text{ if } 7 \mid k
\end{cases}
\\
n &\equiv 4 \pmod{7} \Rightarrow n = 14k+4 =
\begin{cases}
2 \cdot 3^2 + 7 \cdot 2(k-1) & \text{ if } 7 \nmid (k-1) \\
2 \cdot 4^2 + 7 \cdot 2(k-2) & \text{ if } 7 \mid (k-1)
\end{cases}
\end{align*}

\noindent (ii) $n \equiv 1 \pmod{4}$: Consider an integral ternary lattice $N = \qf{1}_\Z \perp
\binlattice2114_\Z$. We know that the class number of $N$ is two \cite{Brandt-Intrau} and $\gen{N}$
represents all positive integers not divisible by $7$. Since $N_2$ is isotropic over $\Z_2$, $N$
represents positive integers $4k$ not divisible by $7$ \cite{miI-96}. Note that $\widetilde{L}$ has
a sublattice $\qf{1, 7}_\Z \perp \binlattice{14}77{28}_\Z = \qf{1}_\Z \perp 7 \left( \qf{1}_\Z
\perp \binlattice2114_\Z \right) = \qf{1}_\Z \perp N^7$. Suppose $n \ge 169$. Then $n \ra \qf{1}_\Z
\perp N^7$ from the following identities.
\begin{align*}
n &\equiv 1 \pmod{7} \Rightarrow n = 28k+1 =
\begin{cases}
1^2 + 7 \cdot 4k & \text{ if } 7 \nmid k \\
13^2 + 7 \cdot 4(k-6) & \text{ if } 7 \mid k
\end{cases}
\\
n &\equiv 2 \pmod{7} \Rightarrow n = 28k+9 =
\begin{cases}
3^2 + 7 \cdot 4k & \text{ if } 7 \nmid k \\
11^2 + 7 \cdot 4(k-4) & \text{ if } 7 \mid k
\end{cases}
\\
n &\equiv 4 \pmod{7} \Rightarrow n = 28k+25 =
\begin{cases}
5^2 + 7 \cdot 4k & \text{ if } 7 \nmid k \\
9^2 + 7 \cdot 4(k-2) & \text{ if } 7 \mid k
\end{cases}
\end{align*}

\noindent (iii) $n \equiv 3 \pmod{8}$: Consider an integral ternary lattice $M = \qf{11}_\Z \perp
\binlattice2114_\Z$. We know that the class number of $M$ is 5 \cite{Brandt-Intrau} and $\gen{M}$
represents all positive integers not divisible by $7$. Since $M_2$ is isotropic over $\Z_2$, $M$
represents positive integers $4^4k$ not divisible by $7$. Note that $\widetilde{L}$ has a
sublattice $\qf{11, 77}_\Z \perp \binlattice{14}77{28}_\Z = \qf{11}_\Z \perp 7 \left( \qf{11}_\Z
\perp \binlattice2114_\Z \right)= \qf{11}_\Z \perp M^7$. There are $a\in\{1,3,5\}$ and $0\le b\le
4^3-1$ such that
\[
    n \equiv 11a^2 \pmod7~  \text{ and }~ n \equiv 11(a+14b)^2 \pmod{4^4}.
\]
Let $k={n-11(a+14b)^2\over7\cdot4^4}$ and $\ell = {n-11(a+14b-7\cdot2^7)^2 \over 7\cdot4^4} =
k-11(a+14b)+11\cdot7\cdot2^6$. Then $k$ and $\ell$ are positive integers if we assume that
$n\ge11(7\cdot2^7)^2=8,830,976$. Note that both $k$ and $\ell$ are not divisible by $7$. Thus $n
\ra \qf{11}_\Z \perp M^7$.
\begin{align*}
n =
\begin{cases}
    11(a+14b)^2 + 7\cdot4^4k & \text{ if } 7 \nmid k \\
    11(a+14b-7\cdot2^7)^2 + 7\cdot4^4\ell & \text{ if } 7 \mid k
\end{cases}
\end{align*}

\noindent (iv) $n \equiv 7 \pmod8$: Consider an integral ternary lattice $R = \binlattice2114_\Z
\perp \qf{23}_\Z$. Then $\widetilde{L}$ contains $\qf{23, 161}_\Z \perp \binlattice{14}77{28}_\Z =
\qf{23}_\Z \perp R^7$. We know that the class number of $R$ is 9 and the genus of $R$ consists of
$9$ lattices \cite{Brandt-Intrau}
\begin{align*}
    & R_1 = R = \binlattice2114_\Z \perp \qf{23}_\Z,
      R_2 = \qf{1, 7, 23}_\Z,
      R_3 = \binlattice4116_\Z \perp \qf{7}_\Z, \\
    & R_4 = \binlattice211{12}_\Z \perp \qf{7}_\Z,
      R_5 = \binlattice3118_\Z \perp \qf{7}_\Z,
      R_6 = \terlattice31114010{15}_\Z, \\
    & R_7 = \qf{1}_\Z \perp \binlattice{11}22{15}_\Z,
      R_8 = \terlattice41114010{11}_\Z,
      R_9 = \qf{1}_\Z \perp \binlattice{14}77{15}_\Z.
\end{align*}
Let $f_{R_i}$ be the quadratic form corresponding to $R_i$. Then $f_R (x,y,z) = 2x^2 + 2xy + 4y^2 +
23z^2$. Note that $\gen{R}$ represents all positive integers not divisible by $7$. If $7 \nmid k$,
then $4^7k \ra R$ from the following identities:
\begin{gather*}
    f_R(4y,-x-y,2z) = 4 f_{R_2}(x,y,z), \\
    f_R(8z,-4x-y-2z,2y) = 4^2 f_{R_3}(x,y,z), \\
    f_R(16y,-2x-8y-4z,4x) = 4^3 f_{R_4}(x,y,z), \\
    f_R(16y,5x-4y-6z,2x+4z) = 4^3 f_{R_5}(x,y,z), \\
    f_R(16y+32z,10x+7y-20z,4x-2y+8z) = 4^4 f_{R_6}(x,y,z),\\
    f_R(16x-16y+64z,7x+33y-40z,-2x+18y+16z) = 4^5 f_{R_7}(x,y,z),\\
    f_R(48x-64y-32z,-47x-28y+66z,18x+8y+36z) = 4^6 f_{R_8}(x,y,z),\\
    f_R(48x-240y-304z,-47x+123y-65z,18x+70y+14z) = 4^7 f_{R_9}(x,y,z).
\end{gather*}
There are $a\in\{1,3,5\}$ and $0\le b\le4^6-1=4095$ such that
\[
    n \equiv 23a^2 \pmod7~ \text{ and } n \equiv 23(a+14b)^2
    \pmod{4^7}.
\]
Let $k = \frac{n-23(a+14b)^2}{7\cdot4^7}$ and $\ell=\frac{n-23(a+14b-7\cdot4^7)^2}{7\cdot4^7} = k -
23\cdot2(a+14b) + 23\cdot7\cdot4^7$. Then $k$ and $\ell$ are positive integers if we assume that $n
\ge 23(7\cdot2^{14})^2 = 302,526,758,912$. Note that both $k$ and $\ell$ are not divisible by $7$.
Thus $n \ra \qf{23}_\Z \perp R^7$.
\begin{align*}
n =
\begin{cases}
    23(a+14b)^2 + 7\cdot4^7k & \text{ if } 7 \nmid k \\
    23(a+14b-7\cdot4^7)^2 + 7\cdot4^7\ell & \text{ if } 7 \mid k
\end{cases}
\end{align*}

We checked over $n$ represented by $\widetilde{L}$ for all $n \leq 302,526,758,912$ by computer
calculation. Therefore $L$ is regular.
\end{proof}

\Case{II [$m=10$]}

(1) $L=\qf{1,5}$ is regular over $\Q{-10}$.
\begin{proof}
Note that
    \[
    H(\gen L) = \{ n \in \N_0 \mid n \equiv 0, 1, 4 \pmod{5} \}.
    \]
The associated quadratic lattice is
    \[
    x\conj{x} + 5y\conj{y} = x_1^2 + 10 x_2^2 + 5 y_1^2 + 50 y_2^2.
    \]
It has a regular sublattice $\ell = x_1^2 + 10 x_2^2 + 5 y_1^2$ \cite{wJ-iK-aS-97}, which
represents all positive integers except $n = 5^e u_{-}$. Now suppose $n=5^2(5k+2)$ or
$n=5^2(5k+3)$, for some integer $k$. Since $n-50\cdot1^2$ or $n-50\cdot2^2$ is represented by
$\ell$, hence $n \ra L$ when $n \geq 50 \cdot 2^2$. It is an easy work to check $n \ra L$ for $n <
50 \cdot 2^2$. Therefore $L$ is regular.
\end{proof}

\Case{II [$m=11$]}

(1) $L=\qf{1,4}$ is regular over $\Q{-11}$.
\begin{proof}
Note that
    \[
    H(\gen L) = \{ n \in \N_0 \mid n \equiv 1 \pmod{2} \text{ or } n \equiv 0 \pmod{4}\}.
    \]
Since $L$ contains $4\qf{1,1} = \qf{4,4}$ and $\qf{1,1}$ is universal \cite{agE-aK-97(1)}, $n \ra
L$ for all positive integers $n \equiv 0 \pmod{4}$. Note that the associated quadratic lattice of
$L$ is
\[
    x\conj{x} + 4 y\conj{y} = x_1^2 + x_1 x_2 + 3x_2^2 + 4 y_1^2 + 4 y_1 y_2 + 12 y_2^2.
\]
Since $x_1^2 + x_1 x_2 + 3x_2^2 + 4 y_1^2$ is regular \cite{wJ-iK-aS-97} and represents all
positive integers $n \equiv 1 \pmod{2}$ and $n \not= 11^{d} u_{+}$, these $n$ are represented by
$L$. Suppose $n \equiv 1 \pmod{2}$ and $n=11^du_{+}$. Since $n/11 \ra L$, $n \ra L$. Therefore $L$
is regular.
\end{proof}

(2) $L=\qf{1,11}$ is regular over $\Q{-11}$.
\begin{proof}
Note that
    \[
    H(\gen L) = \{ n \in \N_0 \mid n \equiv 0, 1, 3, 4, 5, 9 \pmod{11} \}.
    \]
Since $L$ contains $11\qf{1,1}= \qf{11,11}$ and $\qf{1,1}$ is universal \cite{agE-aK-97(1)}, $n \ra
L$ for all positive integers $n \equiv 0 \pmod{11}$. Note that the associated quadratic lattice of
$L$ is
\[
    x\conj{x} + 11 y\conj{y} = x_1^2 + x_1 x_2 + 3x_2^2 + 11 y_1^2 + 11 y_1 y_2 + 33 y_2^2.
\]
Since $x_1^2 + x_1 x_2 + 3x_2^2 + 11 y_1^2$ is regular \cite{wJ-iK-aS-97} and represents all
positive integers $n \equiv 1, 3, 4, 5, 9 \pmod{11}$, $L$ represents these all integers.  Therefore
$L$ is regular.
\end{proof}

(3) $L=\qf{1,44}$ is regular over $\Q{-11}$.
\begin{proof}
Note that
    \begin{align*}
    H(\gen L) = &\{ n \in \N_0 \mid n \equiv 1 \pmod{2} \text{ or } n \equiv 0 \pmod{4} \} \\
                &\cap \{ n \in \N_0 \mid n \equiv 0,1,3,4,5,9 \pmod{11} \}.
    \end{align*}
The associated quadratic lattice of $L$ is
    \[
    x_1^2 + x_1x_2 + 3x_2^2 + 44y_1^2 + 44y_1y_2 + 132y_2^2.
    \]
The sublattice $x_1^2 + x_1x_2 + 3x_2^2 + 44y_1^2$ is regular \cite{wJ-iK-aS-97} and it represents
all positive integers $n \equiv 0, 1, 3 \pmod{4}$ and $n \ne 11^e u_{-}$. Since $11\qf{1,4}$ is a
sublattice of $L$ and $\qf{1,4}$ is regular (see Case II [$m=11$] (1)), $L$ represents all positive
integers $n \equiv 0 \pmod{11}$ and $n \equiv 0, 1, 3 \pmod{4}$. Therefore $L$ is regular.
\end{proof}

\Case{II [$m=15$]}

(1) $L=\qf{1,3}$ is regular over $\Q{-15}$.
\begin{proof}
Note that
    \[
    H(\gen L) = \{ n \in \N_0 \mid n \equiv 0, 1 \pmod{3} \}.
    \]
Since $3\left( \qf{1} \perp \binlattice2\omega\comega2 \right)$ is a sublattice of $L$ and $\qf{1}
\perp \binlattice2\omega\comega2$ is universal \cite{Iwabuchi}, $n \ra L$ for all positive integers
$n \equiv 0 \pmod{3}$. Suppose $n \equiv 1 \pmod{3}$.
Then $n \ra \qf{1} \perp \binlattice2\omega\comega2$ by universality. If $n = n_1 + n_2$ such that
$n_1 \ra \qf{1}$ and $n_2 \ra \binlattice2\omega\comega2$, then we have $n_1 \equiv 1 \pmod{3}$ and
$n_2 \equiv 0 \pmod{3}$, because $\qf{1}$ represents $n_1 \equiv 0, 1 \pmod{3}$ and
$\binlattice2\omega\comega2$ represents $n_2 \equiv 0, 2 \pmod{3}$. Since $n_2 = 2x_1^2 + x_1 x_2 +
2x_2^2$ has an integral solution with $x_1+x_2 \equiv 0 \pmod{3}$, $n_2 = 3 \left[
x_1-(1+\omega)\frac{x_1+x_2}{3} \right]\conj{\left[ x_1-(1+\omega)\frac{x_1+x_2}{3} \right]} \ra
\qf{3}$. So $n = n_1 + n_2 \ra \qf{1, 3}$ and $L$ is regular.
\end{proof}

(2) $L=\qf{1,5}$ is regular over $\Q{-15}$.
\begin{proof}
Note that
    \[
    H(\gen L) = \{ n \in \N_0 \mid n \equiv 0, 1, 4 \pmod{5} \}.
    \]
The associated quadratic lattice of $L$ is
\[
    x\conj{x} + 5y\conj{y} = x_1^2 + x_1 x_2 + 4 x_2^2 + 5 y_1^2 + 5 y_1 y_2 + 20 y_2^2.
\]
Since it has a regular sublattice $x_1^2 + x_1 x_2 + 4 x_2^2 + 5 y_1^2$ \cite{wJ-iK-aS-97}, which
represents all positive integers $n \equiv 1, 4 \pmod{5}$, $L$ represents all positive integers $n
\equiv 1, 4 \pmod{5}$. Since $5\left( \qf{1} \perp \binlattice2\omega\comega2 \right)$ is a
sublattice of $L$ and $\qf{1} \perp \binlattice2\omega\comega2$ is universal \cite{Iwabuchi}, $n
\ra L$ for all positive integers $n \equiv 0 \pmod{5}$. Therefore $L$ is regular.
\end{proof}

(3) $L=\binlattice2\omega\comega2 \perp \qf{5}$ is regular over $\Q{-15}$.

\begin{proof}
Note that
    \[
    H(\gen L) = \{ n \in \N_0 \mid n \equiv 0, 2, 3 \pmod{5} \}.
    \]
The associated quadratic lattice $\widetilde{L}$ of $L$
    \[
    2x_1^2 + x_1 x_2 + 2x_2^2 + 5y_1^2 + 5y_1 y_2 + 20y_2^2
    \]
has a regular sublattice $2x_1^2 + x_1 x_2 + 2x_2^2 + 5 y^2$ \cite{wJ-iK-aS-97}, which represents
all positive integers $n \equiv 0,2,3 \pmod{5}$ and $n \ne 3^d u_{-}$. Suppose $n \equiv 0,2,3
\pmod{5}$ and $n = 3^d u_{-}$. We may consider $n=3(15k+5)$, $n=3(15k+11)$ or $n=3(15k+14)$ for
some nonnegative integers $k$. Consider a quadratic sublattice $2x_1^2 + x_1 x_2 + 2x_2^2 + 5 y^2 +
75 z^2$ of $\widetilde{L}$. Then $n-75\cdot1^2$ is not of the form $3^d u_{-}$ and hence it is
represented by $2x_1^2 + x_1 x_2 + 2x_2^2 + 5 y^2$. The representability of  $3\cdot5$, $3\cdot20$,
$3\cdot11$, $3\cdot14$ by $L$ is easily checked. Therefore, $L$ is regular.
\end{proof}

(4) $L=\binlattice2\omega\comega2 \perp 3\binlattice2\omega\comega2$ is regular over $\Q{-15}$.

\begin{proof}
Note that
    \[
    H(\gen L) = \{ n \in \N_0 \mid n \equiv 0, 2 \pmod{3} \}.
    \]
Since $3 \ra \binlattice2\omega\comega2$, $3\left(\qf{1}\perp\binlattice2\omega\comega2\right)$ is
a sublattice of $L$. Since the lattice $\qf{1}\perp\binlattice2\omega\comega2$ is universal
\cite{Iwabuchi}, $L$ represents all positive integers $n \equiv 0 \pmod{3}$. Suppose $n \equiv 2
\pmod{3}$. If $n=n_1+n_2$ such that $n_1 \ra \qf{1}$ and $n_2 \ra \binlattice2\omega\comega2$, then
we have $n_1 \equiv 0,1 \pmod{3}$ and $n_2 \equiv 0,2 \pmod{3}$. Since $n \equiv 2 \pmod{3}$, $n_1
\equiv 0 \pmod{3}$. Then $n_1 = (x_1+x_2 \omega)\conj{(x_1+x_2 \omega)}$ has an integral solution
with $x_1 \equiv x_2 \pmod{3}$. Since $n_1 =
6\alpha\conj{\alpha}+3\omega\alpha\conj{\beta}+3\comega\conj{\alpha}\beta+6\beta\conj{\beta}$ with
$\alpha = \frac{x_1-y_1}{3}$ and $\beta = -\frac{x_1+2x_2}{3}$, $n_1 \ra
3\binlattice2\omega\comega2$ and $L$ is regular.
\end{proof}

(5) $L=\binlattice2\omega\comega2 \perp \qf{9}$ is regular over $\Q{-15}$.

\begin{proof}
Note that
    \[
    H(\gen L) = \{ n \in \N_0 \mid n \equiv 2 \pmod{3} \text{ or } n \equiv 0,3 \pmod{9} \}.
    \]
Since $3\qf{1,3}$ is a sublattice of $L$ and $\qf{1,3}$ is a regular lattice which represents all
positive integers $n \equiv 0,1 \pmod{3}$ (see Case[$m=15$] (1)), $L$ represents all positive
integers $n \equiv 0, 3 \pmod{9}$. Suppose $n \equiv 2 \pmod{3}$. The associated quadratic lattice
$\widetilde{L}$ of $L$
    \[
    x_1^2 + x_1 x_2 + x_2^2 + 9y_1^2 + 9y_1 y_2 + 36 y_2^2
    \]
has a regular sublattice $\ell = x_1^2 + x_1 x_2 + x_2^2 + 9y_1^2$ \cite{wJ-iK-aS-97} and it
represents all positive integers $n$ if $n \ne 5^d u_{-}$. Suppose $n=5u_{-}$. Then $n=5(15k+7),
5(15k+13), 5(45k+33)$ or $5(45k+42)$. The quadratic lattice $\widetilde{L}$ has the sublattice
$x_1^2 + x_1 x_2 + x_2^2 + 9y^2 + 135 z^2$. Then, $5(15k+7)-135\cdot2^2$, $5(15k+13)-135\cdot1^2$,
$5(45k+33)-135\cdot1^2$ and $5(45k+42)-135\cdot2^2$ are all represented by $\ell$ when $n \geq 135
\cdot 2^2$. The representability for $n < 135\cdot2^2$ is easily checked. Therefore $L$ is regular.
\end{proof}

(6) $L=\binlattice2\omega\comega2 \perp \qf{15}$ is regular over $\Q{-15}$.

\begin{proof}
Note that
    \[
    H(\gen L) = \{ n \in \N_0 \mid n \equiv 0, 2 \pmod{3} \text{ and } n \equiv 0,2,3 \pmod{5} \}.
    \]
Since $15\left(\qf{1}\perp\binlattice2\omega\comega2\right)$ is a sublattice of $L$ and
$\qf{1}\perp\binlattice2\omega\comega2$ is universal, $L$ represents all positive integers
divisible by $15$. The associated quadratic lattice $\widetilde{L}$ of $L$ is
    \[
    2x_1^2 + x_1x_2 + 2x_2^2 + 15y_1^2 + 15y_1y_2 + 60y_2^2.
    \]
The sublattice $\ell = 2x_1^2 + x_1x_2 + 2x_2^2 + 15y_1^2$ is regular \cite{wJ-iK-aS-97} and it
represents all positive integers $n \equiv 0, 2 \pmod{3}$ and $n \ne 5^{2s}u_{+}$ with $s \geq 1$.

It is enough to consider positive integers $n$ such that $n \equiv 2 \pmod{3}$ and $n = 5^2u_{+}$.
Suppose $n = 5^2(15k+11)$ or $n = 5^2(15k+14)$. The quadratic lattice $\widetilde{L}$ has a
sublattice $2x_1^2 + x_1x_2 + 2x_2^2 + 15y_1^2 + 225z^2$. Then, $n-225\cdot1^2$ and $n-225\cdot2^2$
are represented by $\ell$ when $n \geq 225\cdot2^2$. The representability for $n < 225\cdot2^2$ is
easily checked. Therefore $L$ is regular.
\end{proof}

\begin{Thm}
There are $68$ binary regular normal Hermitian lattices, including $9$ nondiagonal lattices, up to
isometry over $\Q{-m}$ with positive square-free integers $m$. In Table \ref{tbl:Binary regular
Hermitian Lattices: normal case}, $25$ binary universal Hermitian lattices are marked with $\dag$.
\begin{footnotesize}
\begin{table}[h]
\centering

\begin{tabular}{lll} \hline
$\Q{-m}$  & {\rm binary regular Hermitian lattices ($\univ$: universal)} \\ \hline%
$\Q{-1}$  & $\qf{1,1}^\univ$, $\qf{1,2}^\univ$,
            $\qf{1,3}^\univ$, $\qf{1,4}$, $\qf{1,8}$, $\qf{1,16}$,\\
          & $\binlattice{2}{-1+\omega_{1}}{-1+\comega_{1}}3$,
            $\binlattice{3}{-1+\omega_{1}}{-1+\comega_{1}}6$,
            $\binlattice3113$ \\

$\Q{-2}$  & $\qf{1,1}^\univ$, $\qf{1,2}^\univ$, $\qf{1,3}^\univ$,
            $\qf{1,4}^\univ$, $\qf{1,5}^\univ$,
            $\qf{1,8}$, $\qf{1, 16}$, $\qf{1, 32}$,
            $\binlattice{2}{\omega_{2}}{\comega_{2}}5$\\

$\Q{-3}$  & $\qf{1,1}^\univ$, $\qf{1,2}^\univ$,
            $\qf{1,3}$, $\qf{1,4}$, $\qf{1,6}$, $\qf{1,9}$,$\qf{1,12}$, $\qf{1,36}$, $\qf{2,3}$,\\
          & $\binlattice2112$, $\binlattice2115$,
            $\binlattice{3}{1+\omega_{3}}{1+\comega_{3}}5$,
            $\binlattice{5}{2}{2}{8}$\\

$\Q{-5}$  & $\qf{1,2}^\univ$, $\qf{1}\perp\binlattice2{-1+\omega_{5}}{-1+\comega_{5}}3^\univ$, $\qf{1,8}$, $\qf{1,10}$, $\qf{1,40}$, \\
          & $\qf{1}\perp 5\binlattice2{-1+\omega_{5}}{-1+\comega_{5}}3$,
            $\binlattice{2}{-1+\omega_{5}}{-1+\comega_{5}}3 \perp\qf{4}$,\\
          & $\binlattice{2}{-1+\omega_{5}}{-1+\comega_{5}}3 \perp\qf{5}$,
            $\binlattice{2}{-1+\omega_{5}}{-1+\comega_{5}}3 \perp\qf{20}$ \\

$\Q{-6}$  & $\qf{1}\perp\binlattice{2}{\omega_6}{\comega_6}3^\univ$,
            $\qf{1,3}$,
            $\binlattice2{\omega_6}{\comega_6}3 \perp 3\binlattice2{\omega_6}{\comega_6}3$\\

$\Q{-7}$  & $\qf{1,1}^\univ$, $\qf{1,2}^\univ$, $\qf{1,3}^\univ$, $\qf{1,7}$, $\qf{1,14}$, $\binlattice3{\omega_7}{\comega_7}3$\\

$\Q{-10}$ & $\qf{1}\perp\binlattice{2}{\omega_{10}}{\comega_{10}}5^\univ$, $\qf{1,5}$ \\

$\Q{-11}$ & $\qf{1,1}^\univ$, $\qf{1,2}^\univ$, $\qf{1,4}$, $\qf{1,11}$, $\qf{1,44}$\\

$\Q{-15}$ & $\qf{1} \perp
            \binlattice{2}{\omega_{15}}{\comega_{15}}2^\univ$,
            $\qf{1,3}$,$\qf{1,5}$,
            $\binlattice2{\omega_{15}}{\comega_{15}}2 \perp \qf{5}$, \\
          & $\binlattice2{\omega_{15}}{\comega_{15}}2 \perp 3\binlattice2{\omega_{15}}{\comega_{15}}{2}$,
            $\binlattice2{\omega_{15}}{\comega_{15}}2 \perp \qf{9}$,
            $\binlattice2{\omega_{15}}{\comega_{15}}2 \perp \qf{15}$ \\

$\Q{-19}$ & $\qf{1,2}^\univ$ \\

$\Q{-23}$ & $\qf{1}\perp\binlattice{2}{\omega_{23}}{\comega_{23}}3^\univ$,
            $\qf{1}\perp\binlattice{2}{-1+\omega_{23}}{-1+\comega_{23}}3^\univ$ \\

$\Q{-31}$ & $\qf{1}\perp\binlattice{2}{\omega_{31}}{\comega_{31}}4^\univ$,
            $\qf{1}\perp\binlattice{2}{-1+\omega_{31}}{-1+\comega_{31}}4^\univ$ \\

\hline
\end{tabular}
\caption{Binary regular Hermitian lattices: normal case} %
\label{tbl:Binary regular Hermitian Lattices: normal case}%
\end{table}
\end{footnotesize}
\end{Thm}

\begin{Rmk}
Rokicki listed $55$ binary regular \emph{diagonal} Hermitian normal lattices over $\Q{-m}$
including a candidate $\qf{1,14}$ over $\Q{-7}$ up to isometry. We confirm her list and make up a
list of all binary regular Hermitian normal lattices including (1) four diagonal lattices
$\qf{1}\perp 5\binlattice2{-1+\omega_{5}}{-1+\comega_{5}}3$ over $\Q{-5}$,
$\binlattice2{\omega_{6}}{\comega_{6}}3 \perp 3\binlattice2{\omega_{6}}{\comega_{6}}3$ over
$\Q{-6}$, and $\binlattice2{\omega_{15}}{\comega_{15}}2 \perp\qf{5}$,
$\binlattice2{\omega_{15}}{\comega_{15}}2 \perp\qf{9}$ over $\Q{-15}$ which she missed, (2) a
universal lattice $\qf{1,3}$ over $\Q{-7}$ uncaught at that time and (3) nine nondiagonal lattices.
Also we provide the complete proof of regularity of each lattice. Some of Rokicki's proofs are
corrected and simplified.
\end{Rmk}

\begin{Rmk}
Among the Hermitian lattices, we have some interesting lattices. For example, a lattice
$\binlattice2{\omega_{6}}{\comega_{6}}3 \perp 3\binlattice2{\omega_{6}}{\comega_{6}}3$ is isometric
to $\binlattice9{4\omega_{6}}{4\comega_{6}}{11}$ over $\Q{-6}$. The former is a direct sum of
nonfree unary lattices and the latter is a binary free lattice. And we have another example of the
same phenomenon. A lattice $\binlattice2{\omega_{15}}{\comega_{15}}2 \perp
3\binlattice2{\omega_{15}}{\comega_{15}}2$ is isometric to
$\binlattice8{-3+4\omega_{15}}{-3+4\comega_{15}}8$ over $\Q{-15}$. On the other hand, although both
$\binlattice2112$ over $\Q{-3}$ and $\binlattice8{-3+4\omega_{15}}{-3+4\comega_{15}}8$ over
$\Q{-15}$ have even diagonal entries, they can represent odd integers. So they are normal lattices.
\end{Rmk}

\begin{Rmk}
The binary subnormal regular Hermitian lattices will be investigated in our next articles. Binary
subnormal regular Hermitian lattices over $\Q{-m}$ with norm ideal $2\Fro$ occur only when
\[
    m = 1, 2, 5, 6, 10, 13, 14, 22, 29, 34, 37 \text{ and } 38.
\]
Also, we know that binary subnormal regular Hermitian lattices exist over $\Q{-m}$ whose norm ideal
is $m\Fro$. For example, $\binlattice3{\sqrt{-3}}{-\sqrt{-3}}3$ over $\Q{-3}$ is a binary subnormal
regular Hermitian lattice with norm ideal $3\Fro$. It is an impossible phenomenon for quadratic
lattices over $\Z$.
\end{Rmk}

%%%%%%%%%%%%%%%%%%%%%%%%%%%%%%%%%%%%%%%%%%%%%%%%%%%%%%%%%%%%%%%%%%%%%%%%%%%%%%%%%%%%%%%%%%%%%%

\end{document}